\documentclass[
reqno,
11pt,
oneside
]{article}

\usepackage[ngerman, english]{babel}
\usepackage[utf8]{inputenc}
\usepackage{enumerate}
\usepackage[normalem]{ulem}
\usepackage[babel,french=guillemets,german=swiss]{csquotes}
\usepackage[center]{caption}
\usepackage{cite}

\usepackage[%
left=1.5in,
right=1.5in,
bottom=1.5in,
top=1in,
footskip=0.5in,
]{geometry}

\usepackage{color}
\definecolor{LinkColor}{rgb}{0,0,1}
\definecolor{LinkColor2}{rgb}{0,0.5,0}
\definecolor{lbcolor}{rgb}{0.85,0.85,0.85}
\definecolor{FrameColor}{rgb}{0.85,0.85,0.85}
\definecolor{darkgreen}{rgb}{0,0.35,0}
\definecolor{purple}{rgb}{0.75,0,0.65}

\usepackage{lineno}

\newcommand*\patchAmsMathEnvironmentForLineno[1]{%
	\expandafter\let\csname old#1\expandafter\endcsname\csname #1\endcsname
	\expandafter\let\csname oldend#1\expandafter\endcsname\csname end#1\endcsname
	\renewenvironment{#1}%
	{\linenomath\csname old#1\endcsname}%
	{\csname oldend#1\endcsname\endlinenomath}}%
\newcommand*\patchBothAmsMathEnvironmentsForLineno[1]{%
	\patchAmsMathEnvironmentForLineno{#1}%
	\patchAmsMathEnvironmentForLineno{#1*}}%
\AtBeginDocument{%
	\patchBothAmsMathEnvironmentsForLineno{equation}%
	\patchBothAmsMathEnvironmentsForLineno{align}%
	\patchBothAmsMathEnvironmentsForLineno{flalign}%
	\patchBothAmsMathEnvironmentsForLineno{alignat}%
	\patchBothAmsMathEnvironmentsForLineno{gather}%
	\patchBothAmsMathEnvironmentsForLineno{multline}%
}

\usepackage{enumitem}

\usepackage{graphicx}
\usepackage{placeins}

\usepackage{amsmath}
\usepackage{amssymb}
\usepackage{dsfont}
\usepackage{empheq}
\usepackage{amsthm}
\numberwithin{equation}{section}

\usepackage[%
pdftitle={Titel},%
pdfauthor={Autor},%
pdfcreator={LaTeX, LaTeX with hyperref and KOMA-Script},
pdfsubject={Betreff}, 
pdfkeywords={Keywords}
]{hyperref} 

\hypersetup{%
	colorlinks	=true,
	linkcolor	=LinkColor,%
	anchorcolor	=LinkColor,%
	citecolor	=LinkColor2,%
	filecolor	=LinkColor,%
	menucolor	=LinkColor,%
	urlcolor	=LinkColor,%
}

\newtheorem{theorem}{Theorem}[section]
\newtheorem{lemma}[theorem]{Lemma}
\newtheorem{proposition}[theorem]{Proposition}
\newtheorem{corollary}[theorem]{Corollary}
\newtheorem{definition}[theorem]{Definition}

\theoremstyle{definition}
\newtheorem{remark}[theorem]{Remark}

\makeatletter
\renewenvironment{proof}[1][\proofname]{%
	\par\pushQED{\qed}\normalfont%
	\topsep6\p@\@plus6\p@\relax
	\trivlist\item[\hskip\labelsep\bfseries#1\@addpunct{.}]%
	\ignorespaces
}{%
	\popQED\endtrivlist\@endpefalse
}
\makeatother

\makeatletter
\renewcommand\paragraph{\@startsection{paragraph}{4}{\z@}%
	{1ex \@plus1ex \@minus.2ex}%
	{-1em}%
	{\normalfont\normalsize\bfseries}}
\renewcommand\subparagraph{\@startsection{paragraph}{4}{\z@}%
	{1ex \@plus1ex \@minus.2ex}%
	{-1em}%
	{\normalfont\normalsize\itshape}}
\makeatother

\newcommand{\norm}[1]{\ensuremath\left\| #1 \right\|}
\newcommand{\abs}[1]{\ensuremath\left|#1 \right|}
\newcommand{\meano}[1]{\ensuremath\left< #1 \right>_\Omega}
\newcommand{\meang}[1]{\ensuremath\left< #1 \right>_\Gamma}
\newcommand{\inn}[2]{\ensuremath\left( #1 \hspace{1pt}{,}\hspace{1pt} #2 \right)}
\newcommand{\biginn}[2]{\ensuremath\big( #1 \hspace{1pt}{,}\hspace{1pt} #2 \big)}
\newcommand{\ang}[2]{\ensuremath\left< #1 \hspace{1pt}{,}\hspace{1pt} #2 \right>}
\newcommand{\bigang}[2]{\ensuremath\big\langle #1 \hspace{1pt}{,}\hspace{1pt} #2 \big\rangle}

\def\R{\mathbb R}
\def\N{\mathbb N}

\def\CC{\mathcal C}
\def\SS{\mathcal S}
\def\FF{\mathcal F}

\def\HH{\mathcal{H}}
\def\VV{\mathcal{V}}
\def\WW{\mathcal{W}}

\def\abso{|\Omega|}
\def\absg{|\Gamma|}

\def\A{{\alpha}}
\def\B{{\beta}}

\def\KA{{K,\alpha}}

\def\KAB{{K,\alpha,\beta}}
\def\KAA{{K,\alpha,\alpha}}
\def\KABS{{K,\alpha,\beta,\ast}}

\def\LB{{L,\beta}}
\def\LBA{{L,\beta,\alpha}}
\def\LBAS{{L,\beta,\alpha,\ast}}

\def\KALB{{K,\alpha,L,\beta}}

\def\mo{{-1}}

\def\n{\mathbf n}

\def\intO{\int_\Omega}

\def\intG{\int_\Gamma}

\def\dx{\;\mathrm dx}

\def\dS{\;\mathrm dS}

\def\delt{\partial_{t}}

\def\deln{\partial_\n}

\def\Grad{\nabla}
\def\Gradg{\nabla_\Gamma}
\def\Lap{\Delta}
\def\Lapg{\Delta_\Gamma}

\def\kernel{\textnormal{ker}}

\def\wto{\rightharpoonup}

\def\emb{\hookrightarrow}

\def\ov{\overline}

\newcommand{\suchthat}{\;\ifnum\currentgrouptype=16 \middle\fi|\;}

\begin{document}
	
\title{\bfseries 
	On second-order and fourth-order 
	elliptic \\
	systems consisting of bulk and surface PDEs: \\  
	Well-posedness, regularity theory and eigenvalue problems}

\author{Patrik Knopf \footnotemark[1] \and Chun Liu \footnotemark[2]}

\date{ }

\renewcommand{\thefootnote}{\fnsymbol{footnote}}
\footnotetext[1]{Department of Mathematics, University of Regensburg, 93053, Germany 
	\tt(\href{mailto:Patrik.Knopf@ur.de}{Patrik.Knopf@ur.de})}
\footnotetext[2]{Department of Applied Mathematics, Illinois Institute of Technology, Chicago, IL 60616, USA 
	\tt (\href{mailto:cliu124@iit.edu}{cliu124@iit.edu})}

\maketitle

%
%

\begin{center}
	\textit{This is a preprint version of the paper. Please cite as:} \\[1ex] 
	P.~Knopf and C.~Liu, \textit{Interfaces Free Bound.},  23:4, 507--533, 2021. \\
	\url{https://doi.org/10.4171/IFB/463}
\end{center}

\medskip

\begin{abstract}
	In this paper, we study second-order and fourth-order elliptic problems which include not only a Poisson equation in the bulk but also an inhomogeneous Laplace--Beltrami equation on the boundary of the domain. The bulk and the surface PDE are coupled by a boundary condition that is either of Dirichlet or Robin type. We point out that both the Dirichlet and the Robin type boundary condition can be handled simultaneously through our formalism without having to change the framework. Moreover, we investigate the eigenvalue problems associated with these second-order and fourth-order elliptic systems. We further discuss the relation between these elliptic problems and certain parabolic problems, especially the Allen--Cahn equation and the Cahn--Hilliard equation with dynamic boundary conditions.
	
	
\end{abstract}

\setlength\parindent{0ex}
\setlength\parskip{1ex}

\bigskip

%
%
\newpage

\section{Introduction}

In this paper, $\Omega$ denotes a bounded domain in $\R^d$ (with $d\in\N$, $d\ge 2$) whose boundary is denoted by $\Gamma:=\partial\Omega$ and is supposed to have at least Lipschitz regularity. Moreover, $\n$ denotes the outer unit normal vector field on $\Gamma$.

\paragraph{A second-order problem with bulk-surface coupling.}
We first consider the following second-order elliptic system consisting of a Poisson equation in the bulk and an inhomogeneous Laplace--Beltrami equation on the surface:
\begin{subequations}
	\label{I:SEC}
	\begin{alignat}{3}
	\label{I:SEC:1}
	-\omega \Lap u &= f &&\quad\text{in}\;\Omega, \\
	\label{I:SEC:2}
	-\gamma\Lapg v + \alpha\omega \deln u &= g &&\quad\text{on}\;\Gamma, \\
	\label{I:SEC:3}
	K\deln u &= \alpha v - u &&\quad\text{on}\;\Gamma.
	\end{alignat}
\end{subequations}
The pair $(f,g)$ stand for a generic pair of source terms whose exact properties will be specified in Section~3. Moreover, $\omega,\gamma>0$, $\alpha\in\R$ and $K\ge 0$ are given constants. 
If $\alpha\neq 0$, the equation \eqref{I:SEC:1} in the bulk (i.e., in $\Omega$) and the equation \eqref{I:SEC:2} on the surface (i.e., on $\Gamma$) are coupled through the boundary condition \eqref{I:SEC:3}. 
In the degenerate case $\alpha=0$, the subproblems (\eqref{I:SEC:1},\eqref{I:SEC:3}) and \eqref{I:SEC:2} are completely decoupled.
If $K>0$, \eqref{I:SEC:3} can be regarded as a Robin type boundary condition, which is sometimes also referred to as a Fourier type boundary condition. (It is worth mentioning that from a historical point of view, the term Fourier boundary condition would be more precise as it seems that Robin never used this type of boundary condition himself. We refer to \cite{Gustafson,Gustafson2} for a detailed discussion of this issue.) 
In the case $K=0$, this boundary condition is to be interpreted as the Dirichlet type boundary condition
\begin{align*}
	 u\vert_\Gamma = \alpha v \quad\text{on}\;\Gamma.
\end{align*}
By our approach, both cases $K>0$ and $K=0$ can be handled simultaneously.

For simplicity of the notation and to provide a cleaner presentation, we will set the constants $\omega$ and $\gamma$ to one in the analysis. We will see that the choice $\omega=1$ does not even mean any loss of generality due to a rescaling argument. We establish the existence and uniqueness of weak solutions to \eqref{I:SEC} provided that the source terms belong to suitable spaces. Moreover, we develop a regularity theory for such solutions depending on the regularity of the domain and the source terms.

\paragraph{A second-order eigenvalue problem.}
Associated with \eqref{I:SEC} is the following eigenvalue problem:
\begin{subequations}
	\label{I:SEIG}
	\begin{alignat}{3}
	\label{I:SEIG:1}
	-\omega\Lap u &= \lambda u &&\quad\text{in}\;\Omega, \\
	\label{I:SEIG:2}
	-\gamma\Lapg v + \alpha\omega \deln u &= \lambda v &&\quad\text{on}\;\Gamma, \\
	\label{I:SEIG:3}
	K\deln u &= \alpha v - u &&\quad\text{on}\;\Gamma.
	\end{alignat}
\end{subequations}
It can formally be regarded as a generalization of the
\textit{Wentzell eigenvalue problem}
\begin{subequations}
	\label{WENZ}
	\begin{alignat}{3}
	\label{WENZ:1}
	-\Lap u &= 0 &&\quad\text{in}\;\Omega, \\
	\label{WENZ:2}
	- \gamma\Lapg u + \deln u &= \lambda u &&\quad\text{on}\;\Gamma,
	\end{alignat}
\end{subequations}
or the \textit{Steklov eigenvalue problem}
\begin{subequations}
	\label{STEK}
	\begin{alignat}{3}
	\label{STEK:1}
	-\Lap u &= 0 &&\quad\text{in}\;\Omega, \\
	\label{STEK:2}
	\deln u &= \lambda u &&\quad\text{on}\;\Gamma.
	\end{alignat}
\end{subequations}
In contrast to classical eigenvalue problems, the eigenvalue does not appear in the equation itself but in the boundary condition instead.
After its introduction in \cite{Steklov}, the Steklov eigenvalue problem has already been extensively investigated in the literature from many different perspectives. We refer the reader to  \cite{Banuelos,Belgacem,Brock,Colbois,Escobar,Fraser,Giles,Girouard,Mora} to name but a few. 
There are also several works on the Wentzell eigenvalue problem of which we want to mention \cite{Kennedy,Dambrine,Xia,Du}.

To understand the connection of our system \eqref{I:SEC} to the Wentzell problem and the Steklov problem, we choose $K=0$ and $\alpha=\omega^{-1/2}$ for any $\omega>0$. In particular this means that
$\omega^{1/2}\, u\vert_\Gamma =  v$ on $\Gamma$ due to \eqref{I:SEC:3}.
Multiplying \eqref{I:SEIG:1} by $\omega^{-1}$ and \eqref{I:SEIG:2} by $\omega^{-1/2}$ then yields
\begin{subequations}
	\label{LIM}
\begin{alignat}{3}
	-\Lap u &= \lambda\omega^{-1} u &&\quad\text{in}\;\Omega, \\
	-\gamma\Lapg u + \deln u &= \lambda u &&\quad\text{on}\;\Gamma.
\end{alignat}
\end{subequations}
Now, by formally passing to the limit $\omega\to\infty$, we obtain the Wentzell problem~\eqref{WENZ} as the limit system. Choosing first $\gamma = \omega^\mo$ in \eqref{LIM} and passing to the formal limit $\omega\to\infty$ afterwards, we arrive at the Steklov problem~\eqref{STEK}. 

For the analysis of the eigenvalue problem \eqref{I:SEIG} we will set the constants $\omega$ and $\gamma$ to one again. We prove that there exists a positive unbounded sequence of eigenvalues whose corresponding eigenfunctions form an orthonormal basis of a suitable linear subspace of $L^2(\Omega)\times L^2(\Gamma)$. Moreover, we conclude regularity properties for the eigenfunctions and we show that the eigenvalues can be characterized by a variational minimax principle.

\paragraph{A fourth-order problem with bulk-surface coupling.} 
We next investigate the following fourth-order elliptic problem with bulk-surface coupling:
\begin{subequations}
	\label{I:GEN}
	\begin{alignat}{3}
	\label{I:GEN:1}
	\Lap^2 \phi &= f 
	&&\quad\text{in}\;\Omega, \\
	\label{I:GEN:2}
	\Lapg^2 \psi - \alpha \Lapg\deln \phi - \beta \deln\Lap\phi &= g 
	&&\quad\text{on}\;\Gamma, \\
	\label{I:GEN:3}
	K\, \deln \phi &= \alpha\psi - \phi &&\quad\text{on}\;\Gamma,\\
	\label{I:GEN:4}
	L\, \deln \Lap \phi &= \beta\Lapg \psi - \Lap \phi - \alpha\beta \deln \phi 
	&&\quad\text{on}\;\Gamma.
	\end{alignat}
\end{subequations}
Here, $K,L\ge 0$ and $\alpha,\beta\in\R$ are given constants, and $(f,g)$ denotes a pair of generic source terms whose properties will be specified in Section~5. We further suppose that $\A$ and $\B$ satisfy 
\begin{align*}
	\A\B\abso + \absg \neq 0,
\end{align*}
which will be crucial for the analysis. We will see that the fourth-order system \eqref{I:GEN} can be decoupled into two second-order systems which are both of the type \eqref{I:SEC}: 
\begin{subequations}
	\label{I:ALT}
	\begin{alignat}{3}
	\label{I:ALT:1}
	-\Lap \phi &= \mu  &&\quad\text{in}\;\Omega, \\
	\label{I:ALT:2}
	-\Lapg \psi + \alpha \deln \phi &= \nu &&\quad\text{on}\;\Gamma, \\
	\label{I:ALT:3}
	K\, \deln \phi &= \alpha \psi-\phi &&\quad\text{on}\;\Gamma,\\[1ex]
	\label{I:ALT:4}
	-\Lap \mu &= f &&\quad\text{in}\;\Omega, \\
	\label{I:ALT:5}
	-\Lapg \nu + \beta \deln \mu&= g &&\quad\text{on}\;\Gamma, \\
	\label{I:ALT:6}
	L\, \deln \mu &= \beta \nu-\mu &&\quad\text{on}\;\Gamma.
	\end{alignat}
\end{subequations}
For that reason, the theory developed for the problem \eqref{I:SEC} can be used to establish weak well-posedness and higher regularity for the system \eqref{I:GEN}.

\paragraph{A fourth-order eigenvalue problem.} Inspired by the Steklov eigenvalue problem, also fourth-order eigenvalue problems, in which the eigenvalue appears in the boundary condition, have been extensively investigated in the literature. We refer the reader to \cite{Berchio,Bucur,Bucur-Gazzola,Buoso,Ferrero-Gazzola-Weth,Gazzola-Sweers,Liu,Matevossian} just to mention a few of them. Because of their relation to the Steklov problem, these models are sometimes referred to as \textit{biharmonic Steklov eigenvalue problems}.

In this paper, we study the following eigenvalue problem: 
\begin{subequations}
	\label{I:EIG}
	\begin{alignat}{3}
	\label{I:EIG:1}
	\Lap^2 \phi &= \lambda \phi 
	&&\quad\text{in}\;\Omega, \\
	\label{I:EIG:2}
	\Lapg^2 \psi - \alpha\Lapg\deln \phi - \beta\deln\Lap\phi &= \lambda \psi 
	&&\quad\text{on}\;\Gamma,\\
	\label{I:EIG:3}
	K\, \deln \phi &= \alpha\psi - \phi, 
	&&\quad\text{on}\;\Gamma,\\
	\label{I:EIG:4}
	L\, \deln \Lap \phi &= \beta\Lapg \psi - \Lap \phi - \alpha\beta\deln \phi 
	&&\quad\text{on}\;\Gamma.
	\end{alignat}
\end{subequations}
As stated above, $K,L\ge 0$ and $\alpha,\beta\in\R$ are given constants with $\A\B\abso + \absg \neq 0$.

The novelty of this eigenvalue problem is that it comprises not only a boundary condition but a fourth-order elliptic equation on the surface. It can thus be regarded as a \textit{bulk-surface biharmonic eigenvalue problem}.
In contrast to the fourth-order Steklov type problems mentioned above, the eigenvalue appears both in the Poisson equation \eqref{I:EIG:1} in the bulk and in the Laplace--Beltrami equation \eqref{I:EIG:2} on the surface but not in the coupling conditions \eqref{I:EIG:3} and \eqref{I:EIG:4}.

As in the second-order case, we prove the existence of a positive unbounded sequence of eigenvalues whose associated eigenfunctions form an orthonormal basis of a suitable linear subspace of $(H^1(\Omega))^*\times (H^1(\Gamma))^*$ (with the asterisk indicating the dual space). We further establish regularity properties for the eigenfunctions and we show that the eigenvalues can be characterized by a variational minimax principle.

\paragraph{Relation to elliptic and parabolic problems with dynamic boundary conditions.} 
We further want to mention that the problems studied in this paper are not only interesting from the perspective of pure analysis but can also be used in the treatment of parabolic problems (especially phase-field models) with dynamic boundary conditions.

The second order problem \eqref{I:SEC} is closely related to the Allen--Cahn equation subject to a dynamic boundary condition that is also of Allen--Cahn type:
\begin{subequations}
	\label{I:AC}
	\begin{alignat}{3}
	\label{I:AC:1}
	\delt u -\Lap u &= F'(u) &&\quad\text{in}\;\Omega \times (0,T), \\
	\label{I:AC:2}
	\delt v -\Lapg v + \alpha\deln u &= G'(u) &&\quad\text{on}\;\Gamma \times (0,T), \\
	\label{I:AC:3}
	K\deln u &= \alpha v - u &&\quad\text{on}\;\Gamma \times (0,T), \\
	(u,v)\vert_{t=0} &= (u_0,v_0) &&\quad\text{on}\;\Omega\times\Gamma.
	\end{alignat}
\end{subequations}
In this phase-field model, $u=u(x,t)$ and $v=v(x,t)$ (the so-called phase-field variables) describe the difference in volume fractions of two different materials located in the bulk $\Omega$ and on the surface $\Gamma$, respectively. This means that the functions $u$ and $v$ are expected to attain values close to $1$ or $-1$ in the regions where only one of the materials is present.
To describe phase separation processes, the bulk potential $F$ and the surface potential $G$ usually exhibit a double-well structure with minima at $\pm1$ and a local maximum at $0$. 

In the Dirichlet case ($K=0$), the problem was investigated, for instance, in \cite{Sprekels, Calatroni, Colli-Fukao}. The Robin case ($K>0$) was studied in \cite{Colli-Fukao-Lam,Lam-Wu}. We further refer to \cite{Gal-Grasselli} where a problem similar to \eqref{I:AC} was discussed.
 
In the analysis of models like \eqref{I:AC} a deeper understanding of the elliptic system \eqref{I:SEC} is very beneficial. Although different strategies have been used in the literature to prove well-posedness, the analysis of the second order eigenvalue problem offers a new possibility to approach systems of the type \eqref{I:AC}. Namely, the orthonormal basis of eigenfunctions to the problem \eqref{I:SEIG} can be used to approximate equations like \eqref{I:AC} by means of a Faedo--Galerkin scheme.   

We also want to mention some further works on second order elliptic or parabolic problems subject to dynamic boundary conditions that are related to the elliptic problem \eqref{I:SEC}. In \cite{Vazquez}, the Laplace equation with dynamic boundary conditions of reaction-diffusion type was studied, and in \cite{Meyries}, nonlinear problems with parabolic dynamic boundary conditions were investigated. An overview about certain classes of elliptic and parabolic problems with dynamic boundary conditions of Wentzell type is given in \cite{Gal-review}.

Similar to the second-order case, the fourth-order elliptic problem \eqref{I:GEN} (or its decoupled equivalent \eqref{I:ALT}) is closely related to the Cahn--Hilliard equation subject to a dynamic boundary condition that also exhibits a Cahn--Hilliard structure:
\begin{subequations}
	\label{I:CH}
	\begin{alignat}{3}
	\label{I:CH:1}
	F'(\phi)-\Lap \phi &= \mu  &&\quad\text{in}\;\Omega\times (0,T), \\
	\label{I:CH:2}
	G'(\psi)-\Lapg \psi + \alpha \deln \phi &= \nu &&\quad\text{on}\;\Gamma\times (0,T), \\
	\label{I:CH:3}
	K\, \deln \phi &= \alpha \psi-\phi &&\quad\text{on}\;\Gamma\times (0,T),\\[1ex]
	\label{I:CH:4}
	\delt \phi -\Lap \mu &= 0 &&\quad\text{in}\;\Omega\times (0,T), \\
	\label{I:CH:5}
	\delt \psi -\Lapg \nu + \beta \deln \mu&= 0 &&\quad\text{on}\;\Gamma\times (0,T), \\
	\label{I:CH:6}
	L\, \deln \mu &= \beta \nu-\mu &&\quad\text{on}\;\Gamma\times (0,T),\\[1ex]
	\label{I:CH:6}
	(\phi,\psi)\vert_{t=0} & = (\phi_0,\psi_0) &&\quad\text{on}\;\Omega\times\Gamma.
	\end{alignat}
\end{subequations}
As in the Allen-Cahn equation \eqref{I:AC}, the functions $\phi=\phi(x,t)$ and $\psi=\psi(x,t)$ denote phase-field variables, and $F$ and $G$ denote the bulk and the surface potential, respectively. Usually both $F$ and $G$ exhibit a double-well structure as described above.
Moreover, $\mu=\mu(x,t)$ stands for the chemical potential in the bulk whereas $\nu=\nu(x,t)$ denotes the chemical potential on the surface. 

The system \eqref{I:CH} with $K=L=0$ was introduced and analyzed in \cite{Gal,GMS}. In \cite{LW}, the model \eqref{I:CH} with $K=0$ and $L=\infty$ (meaning $\deln\mu = 0$ on $\Gamma\times(0,T)$) was derived by an energetic variational approach. This system was further generalized in \cite{KL} by also allowing $K>0$. The asymptotic limit $K\to 0$ was also studied in \cite{KL}. The case $K=0$ and $0<L<\infty$ and its asymptotic limits $L\to 0$ and $L\to\infty$ were investigated in \cite{KLLM}. A similar nonlocal Cahn--Hilliard model was proposed and analyzed in \cite{KS}.

In the analysis of these models the second-order elliptic problem \eqref{I:SEC} plays a crucial role. For instance in \cite{GK,KL,KLLM}, where well-posedness of \eqref{I:CH} was established based on a gradient-flow approach, the system \eqref{I:SEC} was essential to define the underlying inner product. However, we point out that the cases $K=0$ (or $L=0$) and $K>0$ (or $L>0$) always had to be handled separately, whereas in this paper we establish a formalism to approach all these cases simultaneously. We are further convinced that the orthonormal basis of eigenfunctions to the second-order problem \eqref{I:SEIG} or the fourth-order problem \eqref{I:EIG} could potentially be used to discretize the system \eqref{I:CH} by a Faedo--Galerkin scheme, which would provide a new approach to tackle such problems.

\section{Notation and preliminaries}

In this Section we introduce some notation and preliminaries that will be used throughout this paper. 

\begin{enumerate}[label = $\mathrm{(P\arabic*)}$, ref = $\mathrm{P\arabic*}$, leftmargin = *]
\item In this paper, $\N$ denotes the set of natural numbers excluding zero, and $\N_0 = \N\cup\{0\}$.
In general, $\Omega$ will denote a bounded domain in $\R^d$ for some $d\in\N$ with $d\ge 2$ whose boundary $\Gamma:=\partial\Omega$ has at least Lipschitz regularity. The case $d=1$ is excluded as the Laplace--Beltrami operator does not make sense on a boundary consisting only of single points. 
\item \label{P:SPACE}
For any Banach space $X$, its norm will be denoted by $\|\cdot\|_X$ and its dual space is denoted by $X^*$. For any $\varphi\in X^*$ and $\zeta\in X$, we write
$\ang{\varphi}{\zeta}_{X}$ to denote their dual pairing. If $X$ is a Hilbert space, its inner product is denoted by $(\cdot,\cdot)_X$.
\item For any $1\le p\le \infty$, $L^p(\Omega)$ and
$L^p(\Gamma)$ stand for the Lebesgue spaces that are equipped with
the standard norms $\|\cdot\|_{L^p(\Omega)}$ and
$\|\cdot\|_{L^p(\Gamma)}$. For $s\ge 0$ and $1\le p\le \infty$, the
symbols $W^{s,p}(\Omega)$ and $W^{s,p}(\Gamma)$ denote the Sobolev
spaces with corresponding norms $\|\cdot\|_{W^{s,p}(\Omega)}$ and
$\|\cdot\|_{W^{s,p}(\Gamma)}$. Note that $W^{0,p}$ can be identified
with $L^p$. All Lebesgue spaces and Sobolev spaces are Banach spaces
and if $p=2$, they are even Hilbert spaces. In this case we will
write $H^s(\Omega)=W^{s,2}(\Omega)$ and
$H^s(\Gamma)=W^{s,2}(\Gamma)$. 

\item Let $C^\infty(\ov\Omega)$ and $C^\infty(\Gamma)$ denote the spaces of smooth functions on $\ov\Omega$ or $\Gamma$, respectively.
For brevity, we will use the notation
\begin{align*}
\CC^\infty&:=C^\infty(\ov\Omega)\times C^\infty(\Gamma).
\end{align*}


\item For any functions $\zeta\in H^1(\Omega)^*$ and $\xi\in H^1(\Gamma)^*$, we define their generalized mean by the duality pairings
\begin{align*}
\meano{\zeta} := \ang{\zeta}{1}_{H^1(\Omega)^*},
\qquad
\meang{\xi} := \ang{\xi}{1}_{H^1(\Gamma)^*}.
\end{align*}
If additionally $\zeta\in L^1(\Omega)$ or $\xi\in L^1(\Gamma)$, the mean can be expressed as
\begin{align*}
	\meano{\zeta} := \frac{1}{\abs{\Omega}} \intO \zeta \dx,
	\qquad
	\meang{\xi} := \frac{1}{\abs{\Gamma}} \intG \xi \dS,
\end{align*}
respectively.

\item For any integer $k\in\N_0$, we introduce the space
\begin{align*}
	\HH^k &:= H^k(\Omega)\times H^k(\Gamma)
\end{align*}
which is endowed with the standard inner product
\begin{align*}
	\biginn{(u,v)}{(\zeta,\xi)}_{\HH^k} 
		:= \biginn{u}{\zeta}_{H^k(\Omega)}
			+ \biginn{v}{\xi}_{H^k(\Gamma)},
	\quad (u,v),(\zeta,\xi) \in \HH^k
\end{align*}
and the induced norm
\begin{align*}
	\norm{(u,v)}_{\HH^k} 
		:= \biginn{(u,v)}{(u,v)}^{1/2}_{\HH^k},
	\quad (u,v) \in \HH^k.
\end{align*}
This means that 
$
	\big(\HH^k,\inn{\cdot}{\cdot}_{\HH^k},\norm{\,\cdot\,}_{\HH^k}\big)
$
is a Hilbert space. 

\item For any $k\in\N_0$, $m\in \N$ and $K\ge 0$, we define the closed linear subspaces
\begin{align*}
\HH_\KA^m &:=\left\{
\begin{aligned}
& \HH^m , &&\text{if}\; K>0,\\
&\big\{ (u,v)\in\HH^m \suchthat 
u\vert_\Gamma = \alpha v \;\;\text{a.e. on}\; \Gamma  
\big\}, &&\text{if}\; K=0.
\end{aligned}
\right.
\\[2ex]
\VV_\B^k &:=\big\{ (u,v)\in \HH^m \suchthat \beta\abs{\Omega}\meano{u} + \abs{\Gamma}\meang{v} = 0 \big\},
\\[2ex]
\WW_{K,\alpha,\beta}^m &:= \HH_\KA^m \cap \VV_\B^m.
\end{align*}
Note that these subspaces are Hilbert spaces with respect to the inner product $\inn{\cdot}{\cdot}_{\HH^r}$ and its induced norm $\norm{\,\cdot\,}_{\HH^r}$ for $r=k$ or $r=m$, respectively.

\item \label{P:INNKA}
Let $K\ge 0$ and $\A\in\R$ be any real numbers. We set
\begin{align*}
	\sigma(K) :=
	\begin{cases}
		K^\mo, &\text{if}\; K>0,\\
		0, &\text{if}\; K=0,
	\end{cases}
\end{align*}
and we define a bilinear form on $\HH^1\times\HH^1$ by
\begin{align*}
\begin{aligned}
	\biginn{(\phi,\psi)}{(\zeta,\xi)}_\KA 
	&:= \intO \Grad\phi\cdot\Grad\zeta \dx
	+ \intG \Gradg\psi\cdot\Gradg\xi \dS \\
	&\qquad + \sigma(K) \intG (\A\psi-\phi)(\A\xi-\zeta) \dS,
\end{aligned}
\end{align*}
for all $(\phi,\psi),(\zeta,\xi) \in \HH^1$.
Moreover, we set
\begin{align*}
	\norm{(\phi,\psi)}_\KA 
		:= \biginn{(\phi,\psi)}{(\phi,\psi)}_\KA^{1/2}.
\end{align*}

Now, let $\B\in\R$ with $\A\B\abso+\absg\neq 0$ be arbitrary.
Then the bilinear form $\inn{\cdot}{\cdot}_\KA$ defines an inner product on $\WW_\KAB^1\,$, and $\norm{\,\cdot\,}_\KA$ defines a norm on $\WW_\KAB^1$ that is equivalent to the norm $\norm{\,\cdot\,}_{\HH^1}$ (see Corollary~\ref{COR:EQU} in the appendix). 

The space 
$$
	\big(\WW_\KAB^1,\inn{\cdot}{\cdot}_\KA,\norm{\,\cdot\,}_\KA\big) 
$$	
is a Hilbert space. Unless stated otherwise, we understand the space $\WW_\KAB^1$ to be standardly endowed with the inner product $\inn{\cdot}{\cdot}_\KA$ and the norm $\norm{\,\cdot\,}_\KA$.
\item \label{P:DUAL} 
For any $\beta\in\R$, we define the space
\begin{align*}
	\VV^\mo_\B := \big\{ (u,v) \in (\HH^1)^* \suchthat \beta\abs{\Omega}\meano{u} + \abs{\Gamma}\meang{v} = 0 \big\}.
\end{align*}
This entails the chain of inclusions
\begin{align*}
	\WW_\KAB^1 \subset \VV^1_\B \subset \VV^\mo_\B \subset (\HH^1)^* \subseteq (\HH_\KA^1)^* 
\end{align*}
for all $K\ge 0$ and $\alpha, \beta\in\R$. 

\end{enumerate}

\medskip

\section{Second-order elliptic problems with bulk-surface \\ coupling of Robin or Dirichlet type} \label{SECT:3}

In this section, we want to investigate the second-order elliptic system \eqref{I:SEC}. For simplicity of the notation and to provide a cleaner presentation, we set $\omega=\gamma=1$. The system \eqref{I:SEC} is thus restated as:
\begin{subequations}
\label{SEC}
	\begin{alignat}{3}
	\label{SEC:1}
	-\Lap u &= f &&\quad\text{in}\;\Omega, \\
	\label{SEC:2}
	-\Lapg v + \alpha \deln u &= g &&\quad\text{on}\;\Gamma, \\
	\label{SEC:3}
	K\deln u &= \alpha v - u &&\quad\text{on}\;\Gamma,
	\end{alignat}
\end{subequations}
where $\alpha\in\R$ and $K\ge 0$ are given constants. 

In fact, the choice $\omega=1$ means no loss of generality due to the following rescaling argument:
Let $\alpha\in\R$, $\omega,\gamma>0$ and $K\ge 0$ be arbitrary and let $(u,v)$ be any solution to the system \eqref{I:SEC}. 
It is then straightforward to check that $(\tilde u,\tilde v) := \omega (u,v)$ is a solution to the system \eqref{I:SEC} with  $\omega$ and $\gamma$ being replaced by $\tilde\omega:=1$ and $\tilde\gamma:=\gamma\omega^{-1}$, respectively. Hence, if the solution $(u,v)$ is known, the solution $(\tilde u,\tilde v)$ can directly be recovered. 

Although it can not be justified by rescaling, we confine ourselves to investigate the problem for $\gamma=1$. We point out that the case $\gamma\neq 1$ can be handled by the same analytical methods. That is, in the case $\gamma\neq 1$, the definition of the inner product $\inn{\cdot}{\cdot}_{K,\alpha}$ would have to be modified slightly (see Remark \ref{REM:WF:SEC}(d)).

As already pointed out in the introduction, for $K>0$, the coupling equation \eqref{SEC:3} can be regarded as the \textit{Robin type boundary condition} 
\begin{align}
	\deln u\vert_\Gamma = \frac 1 K \big(\alpha v - u\big) \quad \text{a.e. on $\Gamma$}.
\end{align}
For $K=0$, \eqref{SEC:3} is to be interpreted as the \textit{Dirichlet type boundary condition} 
\begin{align}
	u\vert_\Gamma = \alpha v \quad \text{a.e. on $\Gamma$}.
\end{align}
Our approach allows to handle both cases simultaneously.

We first consider the system \eqref{SEC} formally and we assume that the functions $u$, $v$, $f$ and $g$ are sufficiently regular. After testing \eqref{SEC:1} and \eqref{SEC:2} with test functions $\zeta$ and $\xi$, respectively, integration by parts leads to the equation
\begin{align}
\label{WF:SEC:F}
\begin{aligned}
&\intO \Grad u \cdot \Grad \zeta \dx
+ \intG \Gradg v \cdot \Gradg \xi \dS
+ \intG \deln u (\alpha\xi-\zeta) \dS\\[1ex]
&\quad = \intO f\zeta \dx + \intG g\xi \dS.
\end{aligned}
\end{align}
Invoking the boundary condition \eqref{SEC:3}, we find that
\begin{align*}
\intG \deln u (\alpha\xi-\zeta) \dS = \sigma(K) \intG (\alpha v-u)(\alpha\xi-\zeta) \dS.
\end{align*}
Hence, in view of \eqref{P:INNKA}, the equation \eqref{WF:SEC:F} can be expressed as
\begin{align*}
\biginn{(u,v)}{(\zeta,\xi)}_\KA = \biginn{(f,g)}{(\zeta,\xi)}_{\HH^0} = \bigang{(f,g)}{(\zeta,\xi)}_{\HH^1_\KA}.
\end{align*}
This motivates the following definition.

\begin{definition} \label{DEF:WS:SEC}
	Let $K\ge 0$ and $\alpha\in\R$ be arbitrary, let $\Omega\subset \R^d$ be a bounded Lipschitz domain 
	and let $(f,g)\in \VV^\mo_\A$ be arbitrary. 
	
	Then a pair $(u,v)\in\HH^1_\KA$ is called a weak solution of the system \eqref{SEC} if the weak formulation
	\begin{align}
	\label{WF:SEC}
	\biginn{(u,v)}{(\zeta,\xi)}_\KA = \bigang{(f,g)}{(\zeta,\xi)}_{\HH^1_\KA}	
	\end{align}
	is satisfied for all test functions $(\zeta,\xi)\in \HH^1_\KA$. 
\end{definition}

\begin{remark} 
	\label{REM:WF:SEC}
	\begin{enumerate}[label = $\mathrm{(\alph*)}$, leftmargin = *]
		\item Suppose that the functions $(u,v)\in\HH^1_\KA$ and $(f,g)\in\VV^\mo_\A$ satisfy the weak formulation \eqref{WF:SEC}. Choosing the test functions $(\zeta,\xi) = (\alpha,1)$ in \eqref{WF:SEC}, we obtain the \textit{compatibility condition}
		\begin{align}
		\label{SEC:COMP}
		\alpha \abs{\Omega}\meano{f} + \abs{\Gamma}\meang{g} = 0.
		\end{align}
		For that reason, this constraint is incorporated in the space of admissible source terms $\VV^\mo_\A$.
		Moreover, in the case $K>0$, we may choose $(\zeta,\xi)=(1,0)$ and $(\zeta,\xi)=(0,1)$. This leads to 
		\begin{align}
		\label{SEC:ID}
		- \abs{\Omega}\meano{f} 
		= \frac 1K \intG (\alpha v- u) \dS,\quad 
		\abs{\Gamma}\meang{g}
		= \alpha\, \frac 1K \intG (\alpha v- u)\dS
		\end{align}
		if $K>0$.
		\item Suppose that $(u,v)\in \HH^1_\KA$ is a weak solution of the system \eqref{SEC} to the source terms $(f,g)\in\VV^\mo_\A$. One can easily see that then the pair 
		$$(u+\alpha c,v+c)\in\HH^1_\KA$$
		is also a weak solution to the source terms $(f,g)$ for any constant $c\in\R$. 
		Hence, in order to discuss unique weak solutions, this constant $c$ needs to be fixed. This can be done, for instance, by demanding that $(u,v)\in\WW_\KAB^1$ for some suitable $\beta\in\R$.
		\item We want to mention that a second-order elliptic equation similar to ours has been investigated in \cite{Elliot-Ranner}. The system studied there reads as follows:
		\begin{subequations}
			\label{ELRA}
			\begin{alignat}{3}
			\label{ELRA:1}
			-\Lap u + u &= f &&\quad\text{in}\;\Omega, \\
			\label{ELRA:2}
			-\Lapg v + v + \deln u &= g &&\quad\text{on}\;\Gamma, \\
			\label{ELRA:3}
			\deln u &= \beta v - \alpha u &&\quad\text{on}\;\Gamma,
			\end{alignat}
		\end{subequations}
		where $\alpha$ and $\beta$ are positive constants. Although we will use similar techniques to tackle the problem \eqref{SEC}, it is not possible to just resort to the results established in \cite{Elliot-Ranner}. 
		For instance, because of the additional terms ``$+u$'' in \eqref{ELRA:1} and ``$+v$'' in \eqref{ELRA:2}, there is no compatibility condition (such as \eqref{SEC:COMP} for our model). 
		\item To investigate the system \eqref{I:SEC} with $\gamma\neq 1$, the inner product $\inn{\cdot}{\cdot}_{K,\alpha}$ would have to be replaced by
		\begin{align*}
		\begin{aligned}
		\biginn{(\phi,\psi)}{(\zeta,\xi)}_{K,\A,\gamma} 
		&:= \intO \Grad\phi\cdot\Grad\zeta \dx
		+ \gamma \intG \Gradg\psi\cdot\Gradg\xi \dS \\
		&\qquad + \sigma(K) \intG (\A\psi-\phi)(\A\xi-\zeta) \dS,
		\end{aligned}
		\end{align*}
		for all $(\phi,\psi),(\zeta,\xi) \in \HH^1$. However, as $\gamma$ is just a positive constant, this modification would not have any crucial impact on the analysis. Thus, the choice $\gamma=1$ is not a real loss of generality. 	
	\end{enumerate}
\end{remark}

We now intend to establish existence and uniqueness as well as regularity results for weak solutions of the system \eqref{SEC}. In view of Remark~\ref{REM:WF:SEC}(b), we require that the weak solution belongs to $\WW^1_\KAB$ for any given $\beta\in\R$ with $\A\B\abso+\absg\neq 0$. This is stated by the following theorem.

\begin{theorem}
	\label{THM:SEC}
	Let $K\ge 0$ and $\alpha\in\R$ be arbitrary and let $\Omega\subset \R^d$  be a bounded Lipschitz domain. Then the following holds:
	\begin{enumerate}[label = $\mathrm{(\alph*)}$, leftmargin = *]
		\item For any $\B\in\R$ with $\A\B\abso+\absg\neq 0$ and any pair of source terms $(f,g)\in\VV^\mo_\A$, there exists a unique weak solution $(u_{(f,g)},v_{(f,g)})\in\WW_\KAB^1$ of the system \eqref{SEC}. 
		
		This means, we can define a solution operator
		\begin{align}
		\label{DEF:S}
		\begin{aligned}
			&\SS_\KAB = \big(\SS_\KAB^\Omega,\SS_\KAB^\Gamma\big):\VV^\mo_\A \to \WW_\KAB^1 \subset \VV^\mo_\B, \\ 
			&\SS_\KAB(f,g) := (u_{(f,g)},v_{(f,g)})
		\end{aligned}
		\end{align}
		mapping any pair of source terms $(f,g)\in \VV^\mo_\A$ onto the corresponding weak solution $(u_{(f,g)},v_{(f,g)})\in \WW_\KAB^1$ of the system \eqref{SEC}.
		
		Moreover, it holds that $\SS_\KAB$ is injective and continuous with
		\begin{align}
			\label{REG:H1}
			\norm{\SS_\KAB(f,g)}_{\HH^1} 
			\le C\norm{(f,g)}_{(\HH^1_\KA)^*}
		\end{align}
		for a constant $C\ge 0$ depending only on $\Omega$, $K$, $\alpha$ and $\beta$.
		\item Suppose that $\Omega$ is of class $C^{k+2}$ and that $(f,g)\in\VV^k_\A$ for any $k\in\N_0$. 
		Then it holds that $\SS_\KAB(f,g) \in \HH^{k+2}_\KA$ with
		\begin{align}
		\label{REG:HK}
		\norm{\SS_\KAB(f,g)}_{\HH^{k+2}}
		\le C \norm{(f,g)}_{\HH^k}	
		\end{align}
		for a constant $C\ge 0$ depending only on $\Omega$, $K$, $\alpha$, $\beta$ and $k$.
		
		This means that $\SS_\KAB(f,g)$ is a strong solution of the system \eqref{SEC}, i.e., the equations of \eqref{SEC} are satisfied (at least) almost everywhere in $\Omega$ and on $\Gamma$, respectively. 
		\item Suppose that $\Omega$ is of class $C^\infty$ and that $(f,g)\in \VV_\A^m$ for every $m\in\N$. 
		Then it holds that $\SS_\KAB(f,g)\in \CC^\infty$.
	\end{enumerate}
\end{theorem}

\begin{proof}
In this proof, let $C>0$ denote generic constants depending only on $\Omega$, $K$, $\alpha$ and $\beta$.

\textit{Proof of $\mathrm{(a)}$.}
Recall that $\WW^1_\KAB$ is standardly endowed with the inner product $\inn{\cdot}{\cdot}_\KA$ and its induced norm $\norm{\,\cdot\,}_\KA$,
and that
\begin{align*}
	(f,g) \in \VV^\mo_\A \subset (\HH^1_\KA)^*.
\end{align*}
Invoking Corollary~\ref{COR:EQU}, we obtain the estimate
\begin{align}
\label{EST:DUAL}
	\bigang{(f,g)}{(\bar\zeta,\bar\xi)}_{\HH^1_\KA} 
	&\le \norm{(f,g)}_{(\HH^1_\KA)^*} \norm{(\bar\zeta,\bar\xi)}_{\HH^1} \notag\\
	&\le C \norm{(f,g)}_{(\HH^1_\KA)^*} \norm{(\bar\zeta,\bar\xi)}_\KA
\end{align}
for all $(\bar\zeta,\bar\xi)\in \WW^1_\KAB$.
This means that the mapping
\begin{align*}
	\WW^1_\KAB \ni (\bar\zeta,\bar\xi) \mapsto \bigang{(f,g)}{(\bar\zeta,\bar\xi)}_{\HH^1_\KA} \in \R
\end{align*}
defines a continuous linear functional which thus belongs to $(\WW^1_\KAB)^*$.
Hence, the Lax-Milgram theorem implies the existence of a unique pair $(u_{(f,g)},v_{(f,g)})\in\WW^1_\KAB$ such that
\begin{align}
\label{EQ:LM}
	\biginn{(u_{(f,g)},v_{(f,g)})}{(\bar\zeta,\bar\xi)}_\KAB 
	= \bigang{(f,g)}{(\bar\zeta,\bar\xi)}_{\HH^1_\KA}
	\quad
	\text{for all}\; (\bar\zeta,\bar\xi)\in \WW^1_\KAB.
\end{align}
It remains to show that \eqref{EQ:LM} holds true for all test functions in $\HH^1_\KA$. To this end, let $(\zeta,\xi)\in \HH^1_\KA$ be arbitrary. We choose
\begin{align*}
	\bar\zeta := \zeta - \beta c
	\quad\text{and}\quad \bar\xi := \xi - c
	\quad\text{where}\quad c:=	
	\frac{\beta\intO \zeta \dx + \intG\xi\dS}{\beta^2\abs{\Omega} + \abs{\Gamma}}.
\end{align*}
By this construction, we have $(\bar\zeta,\bar\xi)\in \HH^1_\KA$ with
\begin{align*}
	\beta \intO \bar\zeta \dx + \intG \bar\xi \dS
	= \beta \intO \zeta \dx + \intG \xi \dS - (\beta^2\abs{\Omega} + \abs{\Gamma}) c
	= 0.
\end{align*}
This means that $(\bar\zeta,\bar\xi)\in\WW^1_\KAB$. Plugging $(\bar\zeta,\bar\xi)$ into \eqref{EQ:LM} we observe that the constant terms cancel out. Hence, since $(\zeta,\xi)\in\HH^1_\KA$ was arbitrary, we conclude that 
\eqref{EQ:LM} holds true for all test functions $(\zeta,\xi)\in\HH^1_\KA$. This means that $(u_{(f,g)},v_{(f,g)})$ is the unique weak solution of the system \eqref{SEC} in the space $\WW^1_\KAB$.
In particular, the solution operator $\SS_\KAB$ is well-defined.

Testing the weak formulation \eqref{WF:SEC} with $(\zeta,\xi)=(u_{(f,g)},v_{(f,g)})$, and using the estimate \eqref{EST:DUAL} as well as Young's inequality, we conclude that
\begin{align}
\label{EST:KA1}
\begin{aligned}
	\norm{(u_{(f,g)},v_{(f,g)})}_\KA^2 
	&= \bigang{(f,g)}{(u_{(f,g)},v_{(f,g)})}_{\HH^1_\KA} \\
	&\le C \norm{(f,g)}_{(\HH^1_\KA)^*} \norm{(u_{(f,g)},v_{(f,g)})}_{\KA}
\end{aligned}
\end{align}
which proves \eqref{REG:H1}. Thus, (a) is established.

In the following we write $(u,v):=\SS_\KAB(f,g)$ for brevity. The generic constants denoted by $C$ may now also depend on $k$.

\textit{Proof of $\mathrm{(b)}$ in the case $K>0$.} 
We first prove the assertion for $k=0$. Fixing $\zeta=0$ the weak formulation \eqref{WF:SEC} reduces to
\begin{align*}
	\intG \Gradg v \cdot \Gradg \xi \dS
	= \intG g\xi - \frac \alpha K(\alpha v-u) \xi \dS,
	\qquad \xi\in H^1(\Gamma).
\end{align*}
This means that $v$ is a weak solution of the elliptic equation
\begin{align*}
	-\Lapg v = G \;\;\text{on}\;\Gamma
	\qquad\text{with}\qquad G:= g - \frac 1K \alpha(\alpha v- u).
\end{align*}
As $u\in H^1(\Omega)\emb H^{1/2}(\Gamma)$, we know that $G\in L^2(\Gamma)$. Hence, we can apply regularity theory for elliptic equations on submanifolds (see, e.g., \cite[s.\,5,\,Thm\,1.3]{Taylor} and recall that $\Gamma$ is a compact submanifold of class $C^2$ without boundary) to infer that $v\in H^2(\Gamma)$ with
\begin{align*}
	\norm{v}_{H^2(\Gamma)} \le C\norm{G}_{L^2(\Gamma)} + C\norm{v}_{H^1(\Gamma)}
	\le C \norm{g}_{L^2(\Gamma)} + C\norm{(u,v)}_{\HH^1}.
\end{align*}
Proceeding as in \eqref{EST:KA1} and using the equivalence of the norms $\norm{\,\cdot\,}_\KA$ and $\norm{\,\cdot\,}_{\HH^1}$ (see Corollary \ref{COR:EQU}), we conclude that 
\begin{align}
\label{EST:KA}
\begin{aligned}
\norm{(u,v)}_{\HH^1}
\le C\norm{(u,v)}_\KA 
\le C\norm{(f,g)}_{\HH^0}
\end{aligned}
\end{align}
and thus,
\begin{align}
\label{EST:VH2}
\norm{v}_{H^2(\Gamma)} 
\le C \norm{(f,g)}_{\HH^0}.
\end{align}
Now, we choose $\xi=0$ in \eqref{WF:SEC}. This leads to 
\begin{align*}
\intO \Grad u \cdot \Grad \zeta \dx
= \intO f\zeta \dx + \sigma(K) \intG (\alpha v-u) \zeta \dS,
\qquad \zeta\in H^1(\Omega).
\end{align*}
This means that $u$ is a weak solution to the Poisson-Neumann problem
\begin{align*}
	\left\{
	\begin{aligned}
		-\Lap u &= f &&\text{in}\;\Omega,\\
		\deln u &= F &&\text{on}\;\Gamma,\\
	\end{aligned}
	\right.
	\qquad\text{with}\qquad F:= \frac 1 K(\alpha v- u).
\end{align*}
From $u\in H^1(\Omega)\emb H^{1/2}(\Gamma)$ and \eqref{SEC:ID}, it follows that
\begin{align*}
	F \in H^{1/2}(\Gamma) 
	\qquad\text{and}\qquad
	\intG F \dS = -\intO f \dx.
\end{align*}
This allows us to apply regularity theory for Poisson's equation with inhomogeneous Neumann boundary condition (see, e.g., \cite[s.\,5,\,Prop.\,7.7]{Taylor}) to infer that $u\in H^2(\Omega)$ with
\begin{align*}
	\norm{u}_{H^2(\Omega)} \le C \norm{f}_{L^2(\Omega)} + C\norm{u}_{H^1(\Omega)} + C\norm{F}_{H^{1/2}(\Gamma)}.
\end{align*}
Using the continuous embeddings $H^1(\Omega)\emb H^{1/2}(\Gamma)$ and $H^1(\Gamma)\emb H^{1/2}(\Gamma)$, we obtain
\begin{align*}
\norm{F}_{H^{1/2}(\Gamma)} 
\le C\norm{u}_{H^{1/2}(\Gamma)} + C\norm{v}_{H^{1/2}(\Gamma)}
\le C\norm{u}_{H^{1}(\Omega)} + C\norm{v}_{H^{1}(\Gamma)},
\end{align*}
and thus,
\begin{align}
	\label{EST:REG:ROB}
	\norm{u}_{H^2(\Omega)} \le C \norm{f}_{L^2(\Omega)} + C\norm{u}_{H^1(\Omega)} + C\norm{v}_{H^{1}(\Gamma)}.
\end{align}
Combining \eqref{EST:KA}, \eqref{EST:VH2} and \eqref{EST:REG:ROB}, we eventually conclude that
\begin{align*}
	\norm{(u,v)}_{\HH^{2}}
	\le C \norm{(f,g)}_{\HH^0}.	
\end{align*}
This proves the assertion if $k=0$. 

The result for $k>0$ can be established inductively as the regularity results cited above hold true for any integer $k\ge 0$. Assuming that $(u,v)\in\VV_\A^k$ is already established for some $k\ge 0$, we can proceed analogously to the case $k=0$ to conclude that $(u,v)\in\HH^{k+2}$ with
\begin{align*}
\norm{(u,v)}_{\HH^{k+2}}
\le C \norm{(f,g)}_{\HH^k}.	
\end{align*}
This means that (b) is established if $K>0$.

\textit{Proof of $\mathrm{(b)}$ in the case $K=0$.} As in the case $K>0$, we first prove the assertion for $k=0$. Choosing an arbitrary test function $\zeta\in H^1_0(\Omega)$ and fixing $\xi=0$, it obviously holds that $(\zeta,\xi)\in\HH^1_\KA$ since $\zeta\vert_\Gamma = 0 = \alpha\xi$ is satisfied almost everywhere on $\Gamma$. Plugging $(\zeta,\xi)$ into the weak formulation \eqref{WF:SEC}, we infer that
\begin{align*}
\intO \Grad u \cdot \Grad \zeta \dx
= \intO f\zeta \dx .
\end{align*}
In particular, as $\zeta\in H^1_0(\Omega)$ was arbitrary, this holds true for all test functions $\zeta\in C^\infty_c(\Omega)$.
This implies that the distributional derivative $\Lap u$ belongs to $L^2(\Omega)$ and satisfies
\begin{align*}
	-\Lap u = f \quad\text{a.e. in}\;\Omega.
\end{align*}
We further know that $u\vert_\Gamma=\alpha v \in H^1(\Gamma)$. Hence, we can apply elliptic regularity theory for the Poisson--Dirichlet problem (see, e.g., \cite[Thm.~3.2]{Brezzi-Gilardi} or \cite[Thm.~A.2]{Colli-Fukao-Lam}) to conclude that $u\in H^{3/2}(\Omega)$ with
\begin{align}
\label{EST:REG:U}
	\norm{u}_{H^{3/2}(\Omega)} \le C \big( \norm{f}_{L^2(\Omega)} + \norm{v}_{H^1(\Gamma)} \big).
\end{align}
Now, since $\Lap u \in L^2(\Omega)$ and $u\in H^{3/2}(\Omega)$, we can use a variant of the elliptic trace theorem (see, e.g., \cite[Thm.~2.27]{Brezzi-Gilardi} or \cite[Thm.~A.1]{Colli-Fukao-Lam}) to conclude that $\deln u \in L^2(\Gamma)$ with
\begin{align}
	\label{EST:REG:DELNU}
	\norm{\deln u}_{L^2(\Gamma)} \le C\norm{u}_{H^{3/2}(\Omega)}.
\end{align}
 Consequently, integration by parts gives
\begin{align}
\label{EQ:DELNU}
	\intO \Grad u \cdot \Grad \zeta \dx - \intO f\zeta \dx = \intG \deln u \zeta \dS \quad \text{for all $\zeta\in H^1(\Omega)$}.
\end{align}
Let now $\xi\in H^1(\Gamma)$ be arbitrary. According to the inverse trace theorem (see, e.g., \cite[Thm.~4.2.3]{Hsiao-Wendland}) there exists a function $\bar\xi \in H^{3/2}(\Omega)$ such that $\bar\xi\vert_\Gamma = \xi$. We choose $\zeta = \alpha^\mo\bar\xi$ and thus $(\zeta,\xi)\in \HH^1_\KA$. Plugging this pair of test functions into the weak formulation \eqref{WF:SEC} and using the identity \eqref{EQ:DELNU}, we obtain
\begin{align*}
	\intG \Gradg v\cdot\Gradg \xi \dS = \intG (g - \alpha \deln u)\xi \dS.
\end{align*}
As $\xi\in H^1(\Gamma)$ was arbitrary, this implies that $v$ is a weak solution of the elliptic equation
\begin{align*}
-\Lapg v = G \;\;\text{on}\;\Gamma
\qquad\text{with}\qquad G:= g - \alpha\deln u.
\end{align*}
Since $G\in L^2(\Gamma)$, we can apply regularity theory for elliptic equations on submanifolds (see, e.g., \cite[s.\,5,\,Thm\,1.3]{Taylor} and recall that $\Gamma$ is a compact submanifold of class $C^2$ without boundary) to conclude that $v\in H^2(\Gamma)$ with
\begin{align*}
\norm{v}_{H^2(\Gamma)} \le C\norm{G}_{L^2(\Gamma)} 
\le C \norm{g}_{L^2(\Gamma)} + C\norm{\deln u}_{L^2(\Gamma)}.
\end{align*}
Using the estimate \eqref{EST:KA} (which obviously holds true for $K=0$), \eqref{EST:REG:U} and \eqref{EST:REG:DELNU}, we thus get
\begin{align}
\label{EST:REG:V}
\norm{v}_{H^2(\Gamma)} 
\le C \norm{(f,g)}_{\HH^0}.
\end{align}
As $u\vert_\Gamma = \alpha v$ almost everywhere on $\Gamma$, we further deduce that $u\vert_\Gamma \in H^2(\Gamma)$. Recalling that $-\Lap u = f$ almost everywhere in $\Omega$, and invoking
elliptic regularity theory for the Poisson--Dirichlet problem (see, e.g., \cite[Thm.~3.2]{Brezzi-Gilardi} or \cite[Thm.~A.2]{Colli-Fukao-Lam}), we eventually conclude that $u\in H^2(\Omega)$ with
\begin{align}
\label{EST:REG:U:2}
\norm{u}_{H^2(\Omega)} \le C \big( \norm{f}_{L^2(\Omega)} + \norm{v}_{H^2(\Gamma)} \big) \le \norm{(f,g)}_{\HH^0}.
\end{align} 
Hence, in combination with \eqref{EST:REG:V}, the estimate
\begin{align*}
	\norm{(u,v)}_{\HH^2} \le C \norm{(f,g)}_{\HH^0}
\end{align*}
directly follows.
This proves the assertion if $k=0$.

The result for $k>0$ can be established inductively as the regularity results cited above hold true for any integer $k\ge 0$.
Assuming that $(u,v)\in\WW^k_\KAB$ is already established for some $k\ge 0$, we can proceed analogously to the case $k=0$ to conclude that $(u,v)\in\HH^{k+2}$ with
\begin{align*}
\norm{(u,v)}_{\HH^{k+2}}
\le C \norm{(f,g)}_{\HH^k}.	
\end{align*}
This means that (b) is established if $K=0$.

In summary, this completes the proof of (b).

\textit{Proof of $\mathrm{(c)}$.} The claim follows by a simple induction exploiting Sobolev's embedding theorem.

Hence, the proof of Theorem~\ref{THM:SEC} is complete.
\end{proof}

\medskip

\begin{corollary}
	\label{COR:SOL}
	Let $K \ge 0$ and $\alpha\in\R$ be arbitrary, let $\Omega\subset \R^d$ be a bounded Lipschitz domain, and let $\SS_\KAA$ denote the solution operator that was introduced in Theorem~\ref{THM:SEC}(a) (with $\beta:=\alpha$). Then the operator 
	\begin{align}
		\SS_\KAA^0 := \SS_\KAA\vert_{\VV^0_\A}: \VV^0_\A \to \VV^0_\A
	\end{align}
	has the following properties: 
	\begin{enumerate}[label = $\mathrm{(\alph*)}$, leftmargin = *]
		\item $\SS_\KAA^0$ is linear, continuous and compact.
		\item $\SS_\KAA^0$ is injective and thus, it holds that $\kernel(\SS_\KAA^0)=\{(0,0)\}$. 
		\item $\SS_\KAA^0$ is self-adjoint with respect to the inner product $\inn{\cdot}{\cdot}_{\HH^0}$ on $\VV^0_\A$.
	\end{enumerate}
\end{corollary}

\medskip

\begin{proof}
	\textit{Proof of $\mathrm{(a)}$.} It directly follows from Theorem~\ref{THM:SEC}(a) that the operator $\SS_\KAA^0$ is well-defined and linear. Let now $(f,g)\in\VV^0_\A$ be arbitrary. Testing the weak formulation \eqref{WF:SEC} written for $(u,v)=\SS_\KAA^0(f,g)$ with $(\zeta,\xi)=\SS_\KAA^0(f,g)$, and using the Cauchy-Schwarz inequality, we obtain
	\begin{align*}
		\norm{\SS_\KAA^0(f,g)}_\KA^2 
		= \biginn{(f,g)}{\SS_\KAA^0(f,g)}_{\HH^0} 
		\le \norm{(f,g)}_{\HH^0}\norm{\SS_\KAA^0(f,g)}_\KA.
	\end{align*}
	and the continuity of the operator $\SS_\KAA^0$ follows immediately.
	Since 
	\begin{align*}
		\SS_\KAA^0(\VV^0_\A) \subset \WW^1_\KAA ,
	\end{align*} 
	and as the embedding $\WW^1_\KAA\emb\VV^0_\A$ is compact, we conclude that $\SS_\KAA^0$ is a compact operator.
	
	\textit{Proof of $\mathrm{(b)}$.} Theorem~\ref{THM:SEC}(a) states that the operator $\SS_\KAA^0$ is injective and thus, its kernel is trivial. 
	
	\textit{Proof of $\mathrm{(c)}$.} Let $(f_1,g_1), (f_2,g_2)\in \VV^0_\A$ be arbitrary. 
	Recalling the definition of the operator $\SS_\KAA$, we obtain
	\begin{align*}
		&\inn{\SS_\KAA(f_1,g_1)}{(f_2,g_2)}_{\HH^0} = \inn{\SS_\KAA(f_1,g_1)}{\SS_\KAA(f_2,g_2)}_\KA \\
		&\quad = \inn{(f_1,g_1)}{\SS_\KAA(f_2,g_2)}_{\HH^0},
	\end{align*}
	which proves the assertion of (c).
	
	Thus, the proof is complete.
\end{proof}

\medskip

\begin{corollary}
	\label{COR:IP}
	Let $K\ge0$ and $\A,\B \in\R$ with $\A\B\abso+\absg\neq 0$ be arbitrary, and let $\Omega\subset \R^d$  be a bounded Lipschitz domain.
	Then the bilinear form
	\begin{align*}
		&\inn{\cdot}{\cdot}_\KABS: \VV^\mo_\A \times \VV^\mo_\A \to \R,\\
		&\inn{(f_1,g_1)}{(f_2,g_2)}_\KABS := \biginn{\SS_\KAB(f_1,g_2)}{\SS_\KAB(f_1,g_2)}_\KA
	\end{align*}
	defines an inner product on the space $\VV^\mo_\A$. The induced norm
	\begin{align*}
		\norm{(f,g)}_\KABS := \inn{(f,g)}{(f,g)}_\KABS^{1/2}, \quad (f,g) \in \VV^\mo_\A,
	\end{align*}
	is equivalent to the norm $\norm{\,\cdot\,}_{(\HH^1_\KA)^*}$ on $\VV^\mo_\A$ and thus, the space
	\begin{align*}
		\big( \VV^\mo_\A,  \inn{\cdot}{\cdot}_\KABS, \norm{\,\cdot\,}_\KABS\big)
	\end{align*}
	is a Hilbert space.
\end{corollary}

\medskip

\begin{proof}
	The mapping $\inn{\cdot}{\cdot}_\KABS$ is obviously well-defined, bilinear and symmetric. Moreover, it holds that
	\begin{align*}
		\inn{(f,g)}{(f,g)}_\KABS \ge 0 \quad\text{for all}\; (f,g)\in \VV^\mo_\A.
	\end{align*}
	Recalling that $\inn{\cdot}{\cdot}_\KA$ is an inner product on $\WW^1_\KAB$, and that the solution operator $\SS_\KAB$ is linear and injective, we conclude that
	\begin{align*}
		&\inn{(f,g)}{(f,g)}_\KABS = 0 
		\quad\Leftrightarrow\quad \SS_\KAB(f,g) = (0,0) 
		\quad\Leftrightarrow\quad (f,g) = (0,0).
	\end{align*}
	This means that $\inn{\cdot}{\cdot}_\KABS$ is positive definite and thus, it defines an inner product on the space $\VV^\mo_\A$.
	
	To prove the equivalence of the norms $\norm{\,\cdot\,}_\KABS$ and $\norm{\,\cdot\,}_{(\HH^1_\KA)^*}$, let $(f,g)\in\VV^\mo_\A$ be arbitrary, and let $C>0$ denote generic constants depending only on $\Omega$, $K$, $\alpha$ and $\beta$. We first infer from Theorem~\ref{THM:SEC}(a) and Corollary~\ref{COR:EQU} that
	\begin{align*}
		\norm{(f,g)}_\KABS 
		= \norm{\SS_\KAB(f,g)}_\KA 
		\le C \norm{\SS_\KAB(f,g)}_{\HH^1} 
		\le C \norm{(f,g)}_{(\HH^1_\KA)^*}.
	\end{align*}
	Let now $(\zeta,\xi)\in\HH^1_\KA$ with $\norm{(\zeta,\xi)}_{\HH^1} \le 1$ be arbitrary. Recalling the definition of the operator $\SS_\KAB$ and using  Lemma~\ref{LEM:POIN} we get
	\begin{align*}
		&\bigang{(f,g)}{(\zeta,\xi)}_{\HH^1_\KA} 
		= \biginn{\SS_\KAB(f,g)}{(\zeta,\xi)}_\KA
		\le \norm{\SS_\KAB(f,g)}_\KA \norm{(\zeta,\xi)}_\KA \\
		&\quad \le C \norm{\SS_\KAB(f,g)}_\KA \norm{(\zeta,\xi)}_{\HH^1} \le C \norm{\SS_\KAB(f,g)}_\KA 
	\end{align*}
	and thus,
	\begin{align*}
		\norm{(f,g)}_{(\HH^1_\KA)^*} \le C \norm{\SS_\KAB(f,g)}_\KA .
	\end{align*}
	This means that the equivalence of the norms is established and thus, the proof is complete.
\end{proof}

\section{A second-order eigenvalue problem}

For $K\ge 0$, $\alpha\in\R$ and $\lambda\in\R$, we now consider the following second-order eigenvalue problem with bulk-surface coupling of Robin/Dirichlet type:
\begin{subequations}
	\label{SEIG}
	\begin{alignat}{3}
	\label{SEIG:1}
	-\Lap u &= \lambda u &&\quad\text{in}\;\Omega, \\
	\label{SEIG:2}
	-\Lapg v + \alpha \deln u &= \lambda v &&\quad\text{on}\;\Gamma, \\
	\label{SEIG:3}
	K\deln u &= \alpha v - u &&\quad\text{on}\;\Gamma.
	\end{alignat}
\end{subequations}

We immediately notice that for all $\lambda\in\R$ there exists at least one weak solution of the system \eqref{SEIG} in the space $\WW^1_\KAA$, as the pair of null functions $(u,v)=(0,0) \in \WW^1_\KAA$ trivially solves the equations. However, weak solutions of \eqref{SEIG} are generally not unique. In the following we will of course be interested in nontrivial solutions. The following proposition provides some important properties of weak solutions to the problem \eqref{SEIG}.

\begin{proposition} \label{PROP:SEIG}
	Let $K\ge 0$, $\alpha\in\R$ and $\lambda\in\R$ be arbitrary, and let $\Omega\subset \R^d$  be a bounded Lipschitz domain. Then the following holds:
	\begin{enumerate}[label = $\mathrm{(\alph*)}$, leftmargin = *]
		\item Let $(u,v)\in \WW^1_\KAA$ be any weak solution of \eqref{SEIG} in the sense of Definition~\ref{DEF:WS:SEC}. Then $(u,v)$ satisfies the relation
		\begin{align}
			\label{REL:SEIG}
			(u,v) = \SS_\KAA(\lambda u,\lambda v) \quad\text{a.e. in $\Omega$,}
		\end{align}
		as well as the identity
		\begin{align}
			\label{ID:SEIG}
			\norm{(u,v)}_\KA^2 = \lambda \norm{(u,v)}_{\HH^0}^2.
		\end{align}
		\item Suppose that $\Omega$ is of class $C^{k+2}$ for any $k\in\N_0$, 
		and let $(u,v)$ be any weak solution of \eqref{SEIG}. 
		Then it holds that $(u,v)\in\WW^{k+2}_\KAA$ with
		\begin{align*}
		\norm{(u,v)}_{\HH^{k+2}} \le C \lambda\, \norm{(u,v)}_{\HH^k}
		\end{align*}
		for a constant $C\ge 0$ depending only on $\Omega$, $K$, $\alpha$ and $k$.
		
		This means that $(u,v)$ is a strong solution of the eigenvalue problem \eqref{SEIG}.
		\item Suppose that $\Omega$ is of class $C^\infty$, 
		and let $(u,v)$ be any weak solution. 
		Then it holds that $(u,v)\in \CC^\infty$.
	\end{enumerate}
\end{proposition}

\begin{proof}	
	Let $(u,v)\in\WW^1_\KAA$ be any weak solution of the system \eqref{SEIG}, and let us fix $(f,g):=(\lambda u,\lambda v) \in \VV^\mo_\A$. This means that $(u,v)$ is a weak solution of the system \eqref{SEC} to the source terms $(f,g)$. 
	However, according to Theorem~\ref{THM:SEC}(a), $\SS_\KAA\big(f,g\big)$ is the unique weak solution of the problem \eqref{SEC} in the space $\WW^1_\KAA$ to the source terms $(f,g)$. 
	We thus conclude that
	\begin{align*}
		\SS_\KAA\big(\lambda u,\lambda v\big) = \SS_\KAA\big(f,g\big) = (u,v)\quad \text{a.e. in $\Omega$}, 
	\end{align*}
	which proves \eqref{REL:SEIG}. In particular, this means that the theory developed in Theorem~\ref{THM:SEC} can be applied on $(u,v)$. 
	Choosing the test functions $(\zeta,\xi)=(u,v)$ in the weak formulation \eqref{WF:SEC} written for $(f,g)=(\lambda u,\lambda v)$, we directly conclude the identity \eqref{ID:SEIG}.
	Moreover, the regularity assertions (b) and (c) are a direct consequence of the corresponding results stated in Theorem~\ref{THM:SEC}.
\end{proof}	

An eigenvalue of \eqref{SEIG} and its corresponding eigenfunctions are defined as follows: 

\begin{definition}
	\label{DEF:SEIG}
	Let $K\ge 0$, $\alpha\in\R$ and $\lambda\in\R$ be arbitrary, and let $\Omega\subset \R^d$  be a bounded Lipschitz domain.
	
	We call $\lambda\in\R$ an \emph{eigenvalue} if the system \eqref{SEIG} possesses at least one nontrivial weak solution $(u,v)\in\WW^1_\KAA$. 
	In this case, the pair $(u,v)$ is referred to as an \emph{eigenfunction} to the eigenvalue $\lambda$.
\end{definition}

We can easily see that eigenvalues must be strictly positive.

\begin{corollary}
	\label{COR:SPOS}
	Let $K\ge 0$, $\alpha\in\R$ and $\lambda\in\R$ be arbitrary, let $\Omega\subset \R^d$  be a bounded Lipschitz domain,
	and let $\lambda\in\R$ be an eigenvalue. Then it holds that $\lambda>0$.
\end{corollary}

\begin{proof}
	We argue by contradiction and assume that $\lambda \le 0$. Let $(u,v)$ be a corresponding eigenfunction. It then follows from \eqref{ID:SEIG} that 
	\begin{align*}
		\norm{(u,v)}_\KA^2 = 0
	\end{align*}
	which directly implies that $(u,v)=(0,0)$. However, this is a contradiction since eigenfunctions are nontrivial by definition. 
\end{proof}

The eigenvalues of the problem \eqref{SEIG} and their corresponding eigenfunctions can be characterized as follows:

\begin{theorem}
	\label{THM:SEIG}
	Let $K\ge 0$ and $\alpha\in\R$ be arbitrary, and let $\Omega\subset \R^d$  be a bounded Lipschitz domain. Then the following holds:
	\begin{enumerate}[label = $\mathrm{(\alph*)}$, leftmargin = *]	
		\item The problem \eqref{SEIG} has countably many eigenvalues and each of them has a finite-dimensional eigenspace. Repeating each eigenvalue according to its multiplicity, we can write them as a sequence $(\lambda_k)_{k\in\N} \subset \R$ with 
		\begin{align*}
		0 < \lambda_1 \le \lambda_2 \le \lambda_3 \le ... 
		\qquad\text{and}\qquad
		\lambda_k \to \infty \quad\text{as}\; k\to \infty. 
		\end{align*}
		\item There exists an orthonormal basis $\big((u_k,v_k)\big)_{k\in\N}$ of $\VV^0_\A$ with respect to the inner product $\inn{\cdot}{\cdot}_{\HH^0}$ where for every $k\in\N$, the pair $(u_k,v_k)$ is an eigenfunction to the eigenvalue $\lambda_k$.\\[1ex]
		In particular, any pair $(u,v)\in\VV^0_\A$ can be expressed as
		\begin{align*}
		(u,v) = \sum_{k=1}^\infty c_k\, (u_k,v_k)
		\quad\text{with}\quad
		c_k:=\biginn{(u,v)}{(u_k,v_k)}_{\HH^0}, \;\; k\in\N.
		\end{align*}
	\end{enumerate}
\end{theorem}

\begin{proof}
	As the solution operator $\SS^0_\KAA:\VV^0_\A\to\VV^0_\A$ to the problem \eqref{SEC} satisfies the properties established in Corollary~\ref{COR:SOL}, the spectral theorem for compact normal operators (see, e.g., \cite[s.\,12.12]{Alt}) can be applied and proves all assertions. Note that the sequence of eigenvalues is strictly positive due to Corollary~\ref{COR:SPOS}.
\end{proof}

\medskip

Furthermore, the eigenvalues and the corresponding eigenfunctions can be characterized by the following variational principle.

\begin{proposition}
	Let $K\ge 0$ and $\alpha\in\R$ be arbitrary, and let $\Omega\subset \R^d$  be a bounded Lipschitz domain. 
	Moreover, let $(\lambda_k)_{k\in\N}$ denote the sequence of eigenvalues from Theorem~\ref{THM:SEIG}.
	For any $k\in\N$, let $S_{k-1}$ denote the collection of all $(k-1)$-dimensional linear subspaces of $\VV^0_\A$.
	
	Then, for any $k\in\N$, the eigenvalue $\lambda_k$ can be represented by the variational principle
	\begin{align*}
	\lambda_k 
	= \underset{V\in S_{k-1}}{\max} 
	\underset{\substack{(\zeta,\xi)\in V^\bot,\\ \norm{(\zeta,\xi)}_{\HH^0}=1}}{\min}\; \norm{\SS^0_\KAA(\zeta,\xi)}_\KA^2
	\end{align*}		
\end{proposition}

The assertion follows immediately from the minimax principle for self-adjoint operators (see, e.g., \cite[Thm.~6.1.2]{Blanchard-Bruning}).

\section{Fourth-order elliptic problems with bulk-surface coupling of Robin or Dirichlet type}

We now consider the following fourth-order elliptic system with bulk-surface coupling of Robin/Dirichlet type and general source terms $(f,g)$:
\begin{subequations}
	\label{GEN}
	\begin{alignat}{3}
	\label{GEN:1}
	\Lap^2 \phi &= f 
	&&\quad\text{in}\;\Omega, \\
	\label{GEN:2}
	\Lapg^2 \psi - \alpha \Lapg\deln \phi - \beta \deln\Lap\phi &= g 
	&&\quad\text{on}\;\Gamma, \\
	\label{GEN:3}
	K\, \deln \phi &= \alpha\psi - \phi &&\quad\text{on}\;\Gamma,\\
	\label{GEN:4}
	L\, \deln \Lap \phi &= \beta\Lapg \psi - \Lap \phi - \alpha\beta \deln \phi 
	&&\quad\text{on}\;\Gamma.
	\end{alignat}
\end{subequations}
Here $K,L\ge 0$ and $\alpha,\beta\in\R$ are given constants with $\A\B\abso+\absg\neq 0$.

Let us first make some formal considerations. 
Assuming that the solution is sufficiently regular, we can introduce the auxiliary variables
\begin{alignat}{2}
	\label{DEF:MU}
	\mu&:=-\Lap \phi &&\quad\text{in}\;\;\Omega,\\
	\label{DEF:NU}
	\nu&:=-\Lapg \psi + \alpha \deln \phi &&\quad\text{on}\;\;\Gamma.
\end{alignat}
Then the system \eqref{GEN} can be equivalently formulated as
\begin{subequations}
	\label{ALT}
	\begin{alignat}{3}
	\label{ALT:1}
	-\Lap \phi &= \mu  &&\quad\text{in}\;\Omega, \\
	\label{ALT:2}
	-\Lapg \psi + \alpha \deln \phi &= \nu &&\quad\text{on}\;\Gamma, \\
	\label{ALT:3}
	K\, \deln \phi &= \alpha \psi-\phi &&\quad\text{on}\;\Gamma,\\[1ex]
	\label{ALT:4}
	-\Lap \mu &= f &&\quad\text{in}\;\Omega, \\
	\label{ALT:5}
	-\Lapg \nu + \beta \deln \mu&= g &&\quad\text{on}\;\Gamma, \\
	\label{ALT:6}
	L\, \deln \mu &= \beta \nu-\mu &&\quad\text{on}\;\Gamma.
	\end{alignat}
\end{subequations}
We observe that the subsystem \eqref{ALT:4}-\eqref{ALT:6} decouples and that both subsystems \eqref{ALT:1}-\eqref{ALT:3} and \eqref{ALT:4}-\eqref{ALT:6} are of the same type as the second-order system \eqref{SEC}.
Recalling the solution operator of the second order problem \eqref{SEC} that was introduced in Theorem~\ref{THM:SEC}, we can express the pair $(\mu,\nu)$ as
\begin{align}
\label{REP:MUNU}
	(\mu,\nu) = \SS_\LBA(f,g) 
	\in \WW^1_\LBA \subset \VV^\mo_\A.
\end{align}
Consequently, since the pair $(\phi,\psi)$ satisfies the subsystem \eqref{ALT:1}-\eqref{ALT:3}, we infer that
\begin{align}
\label{REP:PHIPSI}
\begin{aligned}
	&\biginn{(\phi,\psi)}{(\zeta,\xi)}_\KA = \biginn{(\mu,\nu)}{(\zeta,\xi)}_{\HH^0} = \biginn{\SS_\LBA(f,g)}{(\zeta,\xi)}_{\HH^0} \\
\end{aligned}
\end{align}
for all $(\zeta,\xi)\in\HH^1_\KA$. 

This motivates the following definition.

\begin{definition} \label{DEF:WS:GEN}
	Let $K,L\ge 0$ and $\alpha,\beta\in\R$ with $\A\B\abso+\absg\neq 0$ be arbitrary, let $\Omega\subset \R^d$ be a bounded Lipschitz domain, 
	and let $(f,g)\in\VV^\mo_\B$ be arbitrary. 
	
	Then a pair $(\phi,\psi)\in\WW^1_\KAB$ is called a weak solution of the system \eqref{GEN} if the weak formulation
	\begin{align}
	\label{WF:GEN}
		\biginn{(\phi,\psi)}{(\zeta,\xi)}_\KA = \biginn{\SS_\LBA(f,g)}{(\zeta,\xi)}_{\HH^0}
	\end{align}
	is satisfied for all test functions $(\zeta,\xi)\in \HH^1_\KA$.
\end{definition}

\medskip

In view of \eqref{REP:MUNU} and \eqref{REP:PHIPSI}, the theory developed in Section~3 can now be used to prove well-posedness and regularity results for solutions of the system \eqref{GEN}.

\begin{theorem}
	\label{THM:GEN}
	Let $K,L\ge 0$ and $\alpha,\beta\in\R$ with $\A\B\abso+\absg\neq 0$ be arbitrary and let $\Omega\subset \R^d$ be a bounded Lipschitz domain. Then the following holds:
	\begin{enumerate}[label = $\mathrm{(\alph*)}$, leftmargin = *]
		\item For any pair of source terms $(f,g)\in\VV^\mo_\B$ there exists a unique weak solution $(\phi_{(f,g)},\psi_{(f,g)})\in\WW^1_\KAB$ of the system \eqref{GEN}.
		
		This means, we can define a solution operator
		\begin{align}
		\label{DEF:F}
		\begin{aligned}
			&\FF_\KALB = (\FF_\KALB^\Omega,\FF_\KALB^\Gamma): \VV^\mo_\B \to \VV^\mo_\B, \\
			\quad
			&\FF_\KALB(f,g) := (\phi_{(f,g)},\psi_{(f,g)}) 
			\end{aligned}
		\end{align}
		mapping any pair of source terms $(f,g)\in \VV^\mo_\B$ onto the corresponding weak solution $(\phi_{(f,g)},\psi_{(f,g)})\in\WW^1_\KAB$ of \eqref{GEN}. In particular, it holds that
		\begin{align*}
			\FF_\KALB = \SS_\KAB\circ\SS_\LBA.
		\end{align*}
		
		Moreover, we obtain the estimate
		\begin{align}
		\label{REG:H1:F}
		\norm{\FF_\KALB(f,g)}_{\HH^1} 
		\le C \norm{(f,g)}_{(\HH^1_\KA)^*}, 
		\end{align}
		for a constant $C\ge 0$ depending only on $\Omega$, $K$, $L$, $\A$ and $\B$.
		\item Suppose that $\Omega$ is of class $C^{k+2}$ for any $k\in\{0,1\}$, and that $(f,g)\in\VV^\mo_\B$.
		Then it even holds that 
			$\FF_\KALB(f,g) \in \WW_\KAB^{k+2}$
		with
		\begin{align*}
			\norm{\FF_\KALB(f,g) }_{\HH^{k+2}} \le C \norm{(f,g)}_{(\HH^1_\KA)^*}\,,
		\end{align*}
		for a constant $C\ge 0$ depending only on $\Omega$, $K$, $L$, $\A$, $\B$ and $k$.
		\item Suppose that $\Omega$ is of class $C^{k+4}$, and that $(f,g)\in\VV^k_\B$ for any $k\in \N_0$.
		Then it even holds that 
		$\FF_\KALB(f,g) \in \WW_\KAB^{k+4}$
		with
		\begin{align*}
		\norm{\FF_\KALB(f,g) }_{\HH^{k+4}} \le C \norm{(f,g)}_{\HH^k},
		\end{align*}
		for a constant $C\ge 0$ depending only on $\Omega$, $K$, $L$, $\A$, $\B$ and $k$.
		
		This means that the pair $(\phi,\psi)$ is a strong solution of the system \eqref{GEN}, i.e., all equations of \eqref{GEN} are satisfied (at least) almost everywhere in $\Omega$ or on $\Gamma$, respectively.
		
		\item Suppose that $\Omega$ is of class $C^\infty$ and that $(f,g)\in\VV_\B^m$ for every $m\in\N$. Then it additionally holds that $\FF_\KALB(f,g)\in\CC^\infty$.
	\end{enumerate}
\end{theorem}

\begin{proof}
	In this proof, let $C>0$ denote generic constants depending only on $\Omega$, $K$, $L$, $\alpha$ and $\beta$.
	
	\textit{Proof of $\mathrm{(a)}$.} Let $(f,g)\in\VV_\B^\mo$ be arbitrary. We set 
	\begin{align*}
		(\phi_{(f,g)},\psi_{(f,g)}) := \big(\SS_\KAB \circ \SS_\LBA\big)(f,g) \in \WW^1_\KAB.
	\end{align*}
	Then, recalling the definition of the operator $\SS_\KAB$, as well as the computations \eqref{REP:MUNU} and \eqref{REP:PHIPSI}, we obtain
	\begin{align*}
		\biginn{(\phi_{(f,g)},\psi_{(f,g)})}{(\zeta,\xi)}_\KA 
		&= \Big(\SS_\KAB\big(\SS_\LBA(f,g)\big) , (\zeta,\xi)\Big)_\KA \\
		&= \biginn{\SS_\LBA(f,g)}{(\zeta,\xi)}_{\HH^0} 
	\end{align*}
	for all $(\zeta,\xi)\in\HH_\KA$. Hence, $(\phi_{(f,g)},\psi_{(f,g)})$ is a weak solution of the system \eqref{GEN} to the source terms $(f,g)$ in the sense of Definition \ref{DEF:WS:GEN}. In particular, this means that $(\phi_{(f,g)},\psi_{(f,g)})$ is a weak solution of the subsystem \eqref{ALT:1}--\eqref{ALT:3} where the source terms are uniquely determined as
	\begin{align*}
		(\mu,\nu) = \SS_\LBA(f,g) \in \WW^1_\LBA.
	\end{align*}
	Hence, we conclude from Theorem~\ref{THM:SEC}(a) that the pair $(\phi_{(f,g)},\psi_{(f,g)}) \in \WW^1_\KAB$ is uniquely determined. This means that the operator $\FF_\KALB$ is well defined and exhibits the decomposition 
	\begin{align*}
		\FF_\KALB = \SS_\KAB\circ\SS_\LBA.
	\end{align*}
	Testing the weak formulation \eqref{WF:GEN} written for $(\phi,\psi)=\FF_\KALB(f,g)$ with $(\zeta,\xi)=\FF_\KALB(f,g)$, and using the Cauchy-Schwarz inequality and Corollary~\ref{COR:EQU}, we obtain the estimate
	\begin{align*}
		\norm{\FF_\KALB(f,g)}^2_\KA 
		&= \biginn{\SS_\LBA(f,g)}{\FF_\KALB(f,g)}^2_{\HH^0} \\
		&\le \norm{\SS_\LBA(f,g)}_{\HH^0} \norm{\SS_\LBA(f,g)}_{\HH^0} \\
		&\le \norm{\SS_\LBA(f,g)}_{\HH^1} \norm{\FF_\KALB(f,g)}_\KA.
	\end{align*}
	Invoking the estimate from Theorem~\ref{THM:SEC}(a), we thus get
	\begin{align*}
		\norm{\FF_\KALB(f,g)}_\KA 
		\le \norm{\SS_\LBA(f,g)}_{\HH^1} \le C \norm{(f,g)}_{(\HH^1_\KA)^*} 
	\end{align*}
	which completes the proof of (a).
	
	In the following we write $(\phi,\psi):=\FF_\KALB(f,g)$ and $(\mu,\nu):=\SS_\LBA(f,g)$ for brevity. The generic constants denoted by $C$ may now also depend on $k$.
	
	\textit{Proof of $\mathrm{(b)}$.}
	Since $\Omega$ is at least of class $C^2$ and $(f,g)\in\VV_\B^\mo$, we infer from Theorem~\ref{THM:SEC}(a) that
	\begin{align*}
		(\mu,\nu) \in \WW^1_\LBA
		\quad\text{with}\quad
		\norm{(\mu,\nu)}_{\HH^1} \le C \norm{(f,g)}_{(\HH^1_\KA)^*}\,.
	\end{align*}
	Since $k+2\le 3$, Theorem~\ref{THM:SEC}(b) further implies that
	\begin{align*}
		(\phi,\psi)\in \WW^{k+2}_\KAB
		\quad\text{with}\quad
		\norm{(\phi,\psi)}_{\HH^{k+2}} 
		\le C \norm{(\mu,\nu)}_{\HH^1} 
		\le C \norm{(f,g)}_{(\HH^1_\KA)^*}.
	\end{align*}
	This proves (b).
	
	\textit{Proof of $\mathrm{(c)}$.}
	Since $\Omega$ is now of class $C^{k+4}$ and $(f,g)\in\VV_\B^{k}$, Theorem~\ref{THM:SEC}(b) implies that 
	\begin{align*}
		(\mu,\nu) \in \WW^{k+2}_\LBA 
		\quad\text{with}\quad
		\norm{(\mu,\nu)}_{\HH^{k+2}} \le C \norm{(f,g)}_{\HH^k}.
	\end{align*}
	and consequently, 
	\begin{align*}
		(\phi,\psi)\in \WW^{k+4}_\KAB
		\quad\text{with}\quad
		\norm{(\phi,\psi)}_{\HH^{k+4}} \le C \norm{(\mu,\nu)}_{\HH^{k+2}} \le C \norm{(f,g)}_{\HH^k}.
	\end{align*} 
	This proves (c).
	
	\textit{Proof of $\mathrm{(d)}$.}
	The assertion follows by a simple induction by means of Sobolev's embedding theorem. This completes the proof of Theorem~\ref{THM:GEN}.
	
	Hence, the proof is complete.
\end{proof}

\bigskip

We can show that the solution operator $\FF_\KALB$ satisfies important properties
which will be essential in the next section where a fourth-order eigenvalue problem based on the system \eqref{GEN} is investigated.

\begin{corollary}
	\label{COR:GEN}
	Let $K,L\ge 0$ and $\alpha,\beta\in\R$ with $\A\B\abso+\absg\neq 0$ be arbitrary, and let $\Omega\subset \R^d$ be a bounded Lipschitz domain. Then the operator 
	\begin{align}
	\FF_\KALB: \VV^\mo_\B \to \VV^\mo_\B
	\end{align}
	introduced in Theorem~\ref{THM:GEN}(a) has the following properties: 
	\begin{enumerate}[label = $\mathrm{(\alph*)}$, leftmargin = *]
		\item $\FF_\KALB$ is linear, continuous and compact.
		\item $\FF_\KALB$ is injective and thus, it holds that $\kernel(\FF_\KALB)=\{(0,0)\}$. 
		\item $\FF_\KALB$ is self-adjoint with respect to the inner product $\inn{\cdot}{\cdot}_\LBAS$ on $\VV^\mo_\B$.
	\end{enumerate}
\end{corollary}

\medskip

\begin{proof}
	\textit{Proof of $\mathrm{(a)}$.} We already know from Theorem~\ref{THM:GEN} that the operator $\FF_\KALB$ is well-defined, linear and continuous. Since 
	\begin{align*}
		\FF_\KALB(\VV^\mo_\B) \subset \WW^1_\KAB ,
	\end{align*} 
	and as the embedding $\WW^1_\KAB\emb\VV^\mo_\B$ is compact, we conclude that $\FF_\KALB$ is a compact operator.
	
	\textit{Proof of $\mathrm{(b)}$.} Since $\SS_\KAB$ and $\SS_\LBA$ are injective according to Theorem~\ref{THM:SEC}(a), and $\FF_\KALB = \SS_\KAB\circ\SS_\LBA$, we conclude that $\FF_\KALB$ is injective. Thus, it is a direct consequence that $\FF_\KALB$ has a trivial kernel.
	
	\textit{Proof of $\mathrm{(c)}$.}
	Let now $(f_1,g_1),(f_2,g_2)\in\VV^\mo_\B$ be arbitrary. We set 
	\begin{align*}
		(\phi_i,\psi_i):=\FF_\KALB(f_i,g_i) \in \WW^1_\KAB \subset \HH^1_\KA,\quad i\in\{1,2\}.
	\end{align*}
	Recalling the definitions of $\SS_\KAB$ and $\SS_\LBA$, and using that $\SS_\KAB$ is self-adjoint with respect to the inner product $\inn{\cdot}{\cdot}_{\HH^0}$, we conclude that
	\begin{align*}
		&\biginn{\FF_\KALB(f_1,g_1)}{(f_2,g_2)}_\LBAS 
			= \biginn{\SS_\LBA(\phi_1,\psi_1)}{\SS_\LBA(f_2,g_2)}_\LB \\
		&\quad = \biginn{(\phi_1,\psi_1)}{\SS_\LBA(f_2,g_2)}_{\HH^0} 
			= \biginn{\SS_\KAB\big(\SS_\LBA(f_1,g_1)\big)}{\SS_\LBA(f_2,g_2)}_{\HH^0} \\
		&\quad = \biginn{\SS_\LBA(f_1,g_1)}{\SS_\KAB\big(\SS_\LBA(f_2,g_2)\big)}_{\HH^0} 
			= \biginn{\SS_\LBA(f_1,g_1)}{(\phi_2,\psi_2)}_{\HH^0} \\
		&\quad = \biginn{\SS_\LBA(f_1,g_1)}{\SS_\KAB(\phi_2,\psi_2)}_\LB
			= \biginn{(f_1,g_1)}{\FF_\KALB(f_2,g_2)}_\LBAS\;.
	\end{align*}
	This proves that $\FF_\KALB$ is self-adjoint with respect to the inner product $\inn{\cdot}{\cdot}_\LBAS$ on $\VV^\mo_\B$. 
	
	Thus, the proof is complete.
\end{proof}

\section{A fourth-order eigenvalue problem}

For $K,L\ge 0$, $\alpha,\beta\in\R$ with $\A\B\abso+\absg\neq 0$ and $\lambda\in\R$, we now consider the following fourth-order eigenvalue problem with bulk-surface coupling of Robin/Dirichlet type:
\begin{subequations}
	\label{EIG}
	\begin{alignat}{3}
	\label{EIG:1}
	\Lap^2 \phi &= \lambda \phi 
	&&\quad\text{in}\;\Omega, \\
	\label{EIG:2}
	\Lapg^2 \psi - \alpha\Lapg\deln \phi - \beta\deln\Lap\phi &= \lambda \psi 
	&&\quad\text{on}\;\Gamma,\\
	\label{EIG:3}
	K\, \deln \phi &= \alpha\psi - \phi, 
	&&\quad\text{on}\;\Gamma,\\
	\label{EIG:4}
	L\, \deln \Lap \phi &= \beta\Lapg \psi - \Lap \phi - \alpha\beta\deln \phi 
	&&\quad\text{on}\;\Gamma.
	\end{alignat}
\end{subequations}

The existence of at least one weak solution in $\WW^1_\KAB$ is trivial, as the pair of null functions $(\phi,\psi)=(0,0) \in \WW^1_\KAB$ obviously solves the equations. In general, weak solutions of \eqref{EIG} are not unique, and we are of course interested in nontrivial solutions. The following proposition provides some important properties of weak solutions.

\begin{proposition} \label{PROP:EIG}
	Let $K,L\ge 0$, $\alpha,\beta\in\R$ with $\A\B\abso+\absg\neq 0$ and $\lambda\in\R$ be arbitrary, and let $\Omega\subset \R^d$  be a bounded Lipschitz domain. Then the following holds:
	\begin{enumerate}[label = $\mathrm{(\alph*)}$, leftmargin = *]
		\item Let $(\phi,\psi)\in \WW^1_\KAB$ be any weak solution of \eqref{EIG} in the sense of Definition~\ref{DEF:WS:GEN}. Then $(\phi,\psi)$ satisfies the relation
		\begin{align}
		\label{REL:EIG}
			(\phi,\psi) = \FF_\KALB(\lambda\phi,\lambda\psi)
			\quad\text{a.e. in $\Omega$,}
		\end{align}
		as well as the identity
		\begin{align}
		\label{ID:EIG}
			\norm{(\phi,\psi)}_\KA = \lambda \norm{\SS_\LBA(\phi,\psi)}_{L,\beta}.
		\end{align}
		\item Suppose that $\Omega$ is of class $C^{k+2}$ for any $k\in\{0,1\}$, 
		and let $(\phi,\psi)$ be any weak solution of \eqref{EIG}. 
		Then it holds that $(\phi,\psi)\in\WW^{k+2}_\KAB$ with
		\begin{align*}
		\norm{(\phi,\psi)}_{\HH^{k+2}} \le C \lambda\, \norm{(\phi,\psi)}_{(\HH^1_\KA)^*}\,,
		\end{align*}
		for a constant $C\ge 0$ depending only on $\Omega$, $K$, $L$, $\A$, $\B$ and $k$.
		\item Suppose that $\Omega$ is of class $C^{k+4}$ for any $k\in\N_0$, 
		and let $(\phi,\psi)$ be any weak solution of \eqref{EIG}. 
		Then it holds that $(\phi,\psi)\in\WW^{k+2}_\KAB$ with
		\begin{align*}
		\norm{(\phi,\psi)}_{\HH^{k+4}} \le C \lambda\, \norm{(\phi,\psi)}_{\HH^k}.
		\end{align*}
		for a constant $C\ge 0$ depending only on $\Omega$, $K$, $L$, $\A$, $\B$ and $k$.
		\item Suppose that $\Omega$ is of class $C^\infty$, 
		and let $(\phi,\psi)$ be any weak solution of \eqref{EIG}. 
		Then it holds that $(\phi,\psi)\in \CC^\infty$.
	\end{enumerate}
\end{proposition}

\medskip

\begin{proof}
	Let $(\phi,\psi)\in\WW^1_\KAB$ be an arbitrary weak solution of the system \eqref{EIG}, and let us fix $(f,g):=(\lambda \phi,\lambda \psi)\in \VV^\mo_\B$. This means that $(\phi,\psi)$ is a weak solution of the system \eqref{GEN} to the source terms $(f,g)$. 
	According to Theorem~\ref{THM:GEN}(a), $\FF_\KALB\big(f,g\big)$ is the unique weak solution of the problem \eqref{GEN} in the space $\WW^1_\KAB$ to the source terms $(f,g)$. 
	We thus infer that
	\begin{align*}
	\FF_\KALB\big(\lambda \phi,\lambda \psi\big) = \FF_\KALB\big(f,g\big) = (\phi,\psi) \quad \text{a.e. in $\Omega$}, 
	\end{align*}
	which verifies \eqref{REL:EIG}. This means that the theory developed in Theorem~\ref{THM:GEN} is applicable for the weak solution $(\phi,\psi)$.
	Testing the weak formulation \eqref{WF:GEN} written for $(f,g)=(\lambda \phi,\lambda \psi)$ with $(\zeta,\xi)=(\phi,\psi)$ and recalling the definition of the solution operator $\SS_\LBA$, we conclude that
	\begin{align*}
		\biginn{(\phi,\psi)}{(\phi,\psi)}_\KA &= \lambda \biginn{\SS_\LBA(\phi,\psi)}{(\phi,\psi)}_{\HH^0} \\
		& = \lambda \biginn{\SS_\LBA(\phi,\psi)}{\SS_\LBA(\phi,\psi)}_\LB,
	\end{align*}
	which proves \eqref{ID:EIG}.
	Moreover, the regularity assertions in (b), (c) and (d) are direct consequences of Theorem~\ref{THM:GEN}(b), (c) and (d), respectively.
\end{proof}

An eigenvalue of \eqref{EIG} and its corresponding eigenfunctions are defined as follows: 

\begin{definition}
	\label{DEF:EIG}
	Let $K,L\ge 0$, $\alpha,\beta\in\R$ with $\A\B\abso+\absg\neq 0$ and $\lambda\in\R$ be arbitrary, and let $\Omega\subset \R^d$  be a bounded Lipschitz domain.
	
	We call $\lambda\in\R$ an \emph{eigenvalue} if the system \eqref{EIG} possesses at least one nontrivial weak solution $(\phi,\psi)\in\WW^1_\KAB$. 
	In this case, the pair $(\phi,\psi)$ is referred to as an \emph{eigenfunction} to the eigenvalue $\lambda$.
\end{definition}

\medskip

We immediately observe that eigenvalues must be strictly positive.

\begin{corollary}
	\label{COR:POS}
	Let $K,L\ge 0$, $\alpha,\beta\in\R$ with $\A\B\abso+\absg\neq 0$ and $\lambda\in\R$ be arbitrary, let $\Omega\subset \R^d$  be a bounded Lipschitz domain,
	and let $\lambda\in\R$ be an eigenvalue. Then it holds that $\lambda>0$.
\end{corollary}

\begin{proof}
	We argue by contradiction and assume that $\lambda \le 0$. Let $(\phi,\psi)$ be a corresponding eigenfunction. It then follows from \eqref{ID:EIG} that 
	\begin{align*}
		\norm{(\phi,\psi)}_\KA^2 = 0
	\end{align*}
	which directly yields $(\phi,\psi)=(0,0)$. Since eigenfunctions are nontrivial this is a contradiction and thus, the assertion is established.
\end{proof}

\bigskip

The eigenvalues of the problem \eqref{EIG} and their corresponding eigenfunctions can be characterized as follows:

\begin{theorem}
	\label{THM:EIG}
	Let $K,L\ge 0$ and $\alpha,\beta\in\R$ with $\A\B\abso+\absg\neq 0$ be arbitrary, and let $\Omega\subset \R^d$  be a bounded Lipschitz domain. Then the following holds:
	\begin{enumerate}[label = $\mathrm{(\alph*)}$, leftmargin = *]	
		\item The problem \eqref{EIG} has countably many eigenvalues and each of them has a finite-dimensional eigenspace. Repeating each eigenvalue according to its multiplicity, we can write them as a sequence $(\lambda_k)_{k\in\N} \subset \R$ with 
		\begin{align*}
		0 < \lambda_1 \le \lambda_2 \le \lambda_3 \le ... 
		\qquad\text{and}\qquad
		\lambda_k \to \infty \quad\text{as}\; k\to \infty. 
		\end{align*}
		\item There exists an orthonormal basis $\big((\phi_k,\psi_k)\big)_{k\in\N}$ of $\VV_\beta^\mo$ with respect to the inner product $\inn{\cdot}{\cdot}_\LBAS$ where for each $k\in\N$, the pair $(\phi_k,\psi_k)$ is an eigenfunction to the eigenvalue $\lambda_k$.\\[1ex]
		In particular, any pair $(\phi,\psi)\in\VV_\beta^\mo$ can be expressed as
		\begin{align*}
		 	(\phi,\psi) = \sum_{k=1}^\infty c_k\, (\phi_k,\psi_k)
		 	\quad\text{with}\quad
		 	c_k:=\biginn{(\phi,\psi)}{(\phi_k,\psi_k)}_\LBAS, \;\; k\in\N.
		\end{align*}
	\end{enumerate}
\end{theorem}

\begin{proof}
	We recall the solution operator $\FF_\KALB:\VV_\beta^\mo\to\VV_\beta^\mo$ to the problem \eqref{GEN} for source terms in $\VV_\beta^\mo$. Due to its properties established in Corollary~\ref{COR:GEN}, the spectral theorem for compact normal operators (see, e.g., \cite[s.\,12.12]{Alt}) can be applied and proves all assertions. We point out that the sequence of eigenvalues is strictly positive according to Corollary~\ref{COR:POS}.
\end{proof}

\medskip

Furthermore, the eigenvalues and the corresponding eigenfunctions can be characterized by the following variational principle.

\begin{proposition}
	Let $K,L\ge 0$ and $\alpha,\beta\in\R$ with $\A\B\abso+\absg\neq 0$ be arbitrary and let $\Omega\subset \R^d$  be a bounded Lipschitz domain. 
	Moreover, let $(\lambda_k)_{k\in\N}$ and $\big((\phi_k,\psi_k)\big)_{k\in\N}$ denote the sequences from Theorem~\ref{THM:EIG}.
	For any $k\in\N$, let $S_{k-1}$ denote the collection of all $(k-1)$-dimensional linear subspaces of $\VV_\beta^\mo$.

	Then, for any $k\in\N$, the eigenvalue $\lambda_k$ can be represented by the variational principle
	\begin{align*}
		\lambda_k 
		= \underset{V\in S_{k-1}}{\max} 
		\underset{\substack{(\zeta,\xi)\in V^\bot,\\ \norm{(\zeta,\xi)}_\LBAS=1}}{\min}\; \norm{\FF_\KALB(\zeta,\xi)}_\KA^2
	\end{align*}	
\end{proposition}

The claim follows directly from the minimax principle for self-adjoint operators (see, e.g., \cite[Thm.~6.1.2]{Blanchard-Bruning}).

\appendix
\section{Appendix}

We present a Poincar\'e type inequality with respect to the norm $\norm{\,\cdot\,}_\KA$ for functions in $\WW^1_\KAB$.

\begin{lemma}
\label{LEM:POIN}
	Let $K\ge 0$ and $\alpha,\beta\in\R$ with $\alpha\beta\abso+\absg\neq 0$ be arbitrary, and let $\Omega\subset\R^d$ be a bounded Lipschitz domain $\Omega$ with boundary $\Gamma$. Then there exists a constant $c_P > 0$ depending only on $K$, $\alpha$, $\beta$ and $\Omega$ such that 
\begin{align}
	\norm{(u,v)}_{\HH^0} \le c_P \norm{(u,v)}_\KA
\end{align}
for all pairs $(u,v)\in\WW^1_\KAB$.
\end{lemma}

\begin{proof}
	We prove the assertion by contradiction. Therefore, we assume that the estimate is false. Consequently, for every $k\in\N$ there exists a pair $(u_k,v_k)\in\WW^1_\KAB$  such that
\begin{align}
	\label{ASS:POIN}
	\norm{(u_k,v_k)}_{\HH^0} > k \norm{(u_k,v_k)}_\KA
\end{align}
Thus, the sequence $(\tilde u_k,\tilde v_k)_{k\in\N}$ defined by
\begin{align*}
	\tilde u_k := \frac{u_k}{\norm{(u_k,v_k)}_{\HH^0}},
	\qquad
	\tilde v_k := \frac{v_k}{\norm{(u_k,v_k)}_{\HH^0}},
	\qquad
	k\in\N
\end{align*}
satisfies 
\begin{align}
\label{PC1}
	(\tilde u_k,\tilde v_k) \in \WW^1_\KAB,
	\qquad
	\norm{(\tilde u_k,\tilde v_k)}_{\HH^0} = 1,
	\qquad
	k\in\N.
\end{align}	
Moreover, \eqref{ASS:POIN} implies that  
\begin{align}
\label{PC2}
	\norm{\Grad \tilde u_k}_{L^2(\Omega)}^2 + \norm{\Gradg \tilde v_k}_{L^2(\Gamma)}^2 
		+ \sigma(K) \norm{\alpha \tilde v_k-\tilde u_k}_{L^2(\Gamma)}^2
	= \norm{(\tilde u_k,\tilde v_k)}_\KA^2 < \frac 1 {k^2},
\end{align}
for all $k\in\N$. In particular, \eqref{PC1} and \eqref{PC2} imply that the sequence $(\tilde u_k, \tilde v_k)_{k\in\N}$ is bounded in $\HH^1$. Hence, according to the Banach-Alaoglu theorem, there exists a pair $(u,v)\in\HH^1$ such that $(\tilde u_k,\tilde v_k)\wto (u,v)$ in $\HH^1$ after extraction of a subsequence. 
It thus follows that
\begin{align*}
	\beta\abs{\Omega} \meano{u} + \abs{\Gamma}\meang{v} = 0.
\end{align*}

From the compact embedding $\HH^1\emb\HH^0$, we deduce that
\begin{align*}
	(\tilde u_k,\tilde v_k) \to (u,v) 
	\qquad
	\text{in}\; \HH^0,
\end{align*}
after another subsequence extraction. In particular, this implies that $\norm{(u,v)}_{\HH^0}= 1$.

If $K>0$, we infer from \eqref{PC2} that
\begin{align}
\label{PC3}
	\alpha \tilde v_k - \tilde u_k \to 0 \quad \text{in}\; L^2(\Gamma) 
	\qquad\text{and thus,}\qquad
	\alpha v - u = 0 \quad\text{a.e. on}\; \Gamma.
\end{align}
This obviously holds true for $K=0$, since then $\alpha \tilde v_k - \tilde u_k=0$ for all $k\in\N$.
Hence, we conclude that $(u,v)\in\WW^1_\KAB$.

As the $\HH^0$-norm is weakly lower semicontinuous, we deduce that
\begin{align*}
	\norm{(\Grad u,\Gradg v)}_{\HH^0}
	\le \underset{k\to\infty}{\lim\inf}\norm{(\Grad \tilde u_k,\Gradg \tilde v_k)}_{\HH^0} 
	\le 0.
\end{align*}
This implies that there exist constants $A,B\in\R$ such that $u=A$ almost everywhere in $\Omega$ and $v=B$ almost everywhere on $\Gamma$. 
It then follows from $(u,v)\in\WW^1_\KAB$ and \eqref{PC3} that
\begin{align*}
	\beta\abs{\Omega} A + \abs{\Gamma} B 
	= \beta\abs{\Omega} \meano{u} + \abs{\Gamma}\meang{v} = 0
	\quad\text{and}\quad
	\alpha B - A = 0.
\end{align*}
Since $\alpha\beta\abso+\absg\neq 0$, we conclude that $A=B=0$ which means that $u=0$ almost everywhere in $\Omega$ and $v=0$ almost everywhere on $\Gamma$. However, this is a contradiction to $\norm{(u,v)}_{\HH^0}= 1$. This proves the assertion.
\end{proof}

\medskip

The following result is a direct consequence of Lemma \ref{LEM:POIN}.

\begin{corollary}\label{COR:EQU}
	Let $K\ge 0$ and $\alpha,\beta\in\R$ with $\alpha\beta\abso+\absg\neq 0$ be arbitrary, and let $\Omega\subset\R^d$ be a bounded Lipschitz domain $\Omega$ with boundary $\Gamma$. Then there exist constants $A,B> 0$ depending only on $K$, $\alpha$, $\beta$ and $\Omega$ such that for all $(u,v)\in\WW^1_\KAB\,$,
	\begin{align}
		\norm{(u,v)}_{\HH^1} \le A \norm{(u,v)}_\KA
		\qquad\text{and}\qquad
		\norm{(u,v)}_\KA \le B \norm{(u,v)}_{\HH^1}.
	\end{align}
	This means that both norms are equivalent.
\end{corollary}

\section*{Acknowledgements}
Patrik Knopf was partially supported by the RTG 2339 \glqq Interfaces,
Complex Structures, and Singular Limits\grqq\ of the German Science Foundation (DFG).
Chun Liu was partially supported by the grant 1759535 of the National Science Foundation (NSF) and the grant 2024246 of the United States–Israel Binational Science Foundation (BSF). The support is gratefully acknowledged.




\footnotesize

\bibliographystyle{plain}
\bibliography{KL}

\begin{thebibliography}{10}

\bibitem{Alt}
H.W. Alt.
\newblock {\em {Linear Functional Analysis - An Application-Oriented
  Introduction}}.
\newblock Springer, London, 2016.

\bibitem{Giles}
G.~Auchmuty.
\newblock Steklov eigenproblems and the representation of solutions of elliptic
  boundary value problems.
\newblock {\em Numer. Funct. Anal. Optim.}, 25(3-4):321--348, 2004.

\bibitem{Banuelos}
R.~Ba\~{n}uelos and T.~Kulczycki.
\newblock The {C}auchy process and the {S}teklov problem.
\newblock {\em J. Funct. Anal.}, 211(2):355--423, 2004.

\bibitem{Belgacem}
F.~Ben~Belgacem and H.~El~Fekih.
\newblock On {C}auchy's problem. {I}. {A} variational {S}teklov-{P}oincar\'{e}
  theory.
\newblock {\em Inverse Problems}, 21(6):1915--1936, 2005.

\bibitem{Berchio}
E.~Berchio, F.~Gazzola, and T.~Weth.
\newblock Critical growth biharmonic elliptic problems under {S}teklov-type
  boundary conditions.
\newblock {\em Adv. Differential Equations}, 12(4):381--406, 2007.

\bibitem{Blanchard-Bruning}
P.~Blanchard and E.~Br\"{u}ning.
\newblock {\em Variational methods in mathematical physics}.
\newblock Texts and Monographs in Physics. Springer-Verlag, Berlin, 1992.
\newblock A unified approach, Translated from the German by Gillian M. Hayes.

\bibitem{Brezzi-Gilardi}
F.~Brezzi and G.~Gilardi.
\newblock Partial differential equations.
\newblock In H.~Kardestuncer and D.H. Norrie, editors, {\em Finite Element
  Handbook}. McGraw-Hill, New-York, 1987.

\bibitem{Brock}
F.~Brock.
\newblock An isoperimetric inequality for eigenvalues of the {S}tekloff
  problem.
\newblock {\em ZAMM Z. Angew. Math. Mech.}, 81(1):69--71, 2001.

\bibitem{Bucur}
D.~Bucur, A.~Ferrero, and F.~Gazzola.
\newblock On the first eigenvalue of a fourth order {S}teklov problem.
\newblock {\em Calc. Var. Partial Differential Equations}, 35(1):103--131,
  2009.

\bibitem{Bucur-Gazzola}
D.~Bucur and F.~Gazzola.
\newblock The first biharmonic {S}teklov eigenvalue: positivity preserving and
  shape optimization.
\newblock {\em Milan J. Math.}, 79(1):247--258, 2011.

\bibitem{Buoso}
D.~Buoso and L.~Provenzano.
\newblock A few shape optimization results for a biharmonic {S}teklov problem.
\newblock {\em J. Differential Equations}, 259(5):1778--1818, 2015.

\bibitem{Calatroni}
L.~Calatroni and P.~Colli.
\newblock Global solution to the {A}llen-{C}ahn equation with singular
  potentials and dynamic boundary conditions.
\newblock {\em Nonlinear Anal.}, 79:12--27, 2013.

\bibitem{Colbois}
B.~Colbois, A.~El~Soufi, and A.~Girouard.
\newblock Isoperimetric control of the {S}teklov spectrum.
\newblock {\em J. Funct. Anal.}, 261(5):1384--1399, 2011.

\bibitem{Colli-Fukao}
P.~Colli and T.~Fukao.
\newblock The {A}llen-{C}ahn equation with dynamic boundary conditions and mass
  constraints.
\newblock {\em Math. Methods Appl. Sci.}, 38(17):3950--3967, 2015.

\bibitem{Colli-Fukao-Lam}
P.~Colli, T.~Fukao, and K.F. Lam.
\newblock {On a coupled bulk--surface {A}llen--{C}ahn system with an affine
  linear transmission condition and its approximation by a {R}obin boundary
  condition}.
\newblock {\em Nonlinear Anal.}, 184:116--147, 2019.

\bibitem{Dambrine}
M.~Dambrine, D.~Kateb, and J.~Lamboley.
\newblock An extremal eigenvalue problem for the {W}entzell-{L}aplace operator.
\newblock {\em Ann. Inst. H. Poincar\'{e} Anal. Non Lin\'{e}aire},
  33(2):409--450, 2016.

\bibitem{Du}
F.~Du, Q.~Wang, and C.~Xia.
\newblock Estimates for eigenvalues of the {W}entzell-{L}aplace operator.
\newblock {\em J. Geom. Phys.}, 129:25--33, 2018.

\bibitem{Elliot-Ranner}
C.M. Elliott and T.~Ranner.
\newblock {Finite element analysis for a coupled bulk–surface partial
  differential equation}.
\newblock {\em IMA Journal of Numerical Analysis}, 33(2):377--402, 2012.

\bibitem{Escobar}
J.F. Escobar.
\newblock A comparison theorem for the first non-zero {S}teklov eigenvalue.
\newblock {\em J. Funct. Anal.}, 178(1):143--155, 2000.

\bibitem{Ferrero-Gazzola-Weth}
A.~Ferrero, F.~Gazzola, and T.~Weth.
\newblock On a fourth order {S}teklov eigenvalue problem.
\newblock {\em Analysis (Munich)}, 25(4):315--332, 2005.

\bibitem{Fraser}
A.~Fraser and R.~Schoen.
\newblock The first {S}teklov eigenvalue, conformal geometry, and minimal
  surfaces.
\newblock {\em Adv. Math.}, 226(5):4011--4030, 2011.

\bibitem{Gal}
C.G. Gal.
\newblock A {C}ahn-{H}illiard model in bounded domains with permeable walls.
\newblock {\em Math. Methods Appl. Sci.}, 29(17):2009--2036, 2006.

\bibitem{Gal-review}
C.G. Gal.
\newblock The role of surface diffusion in dynamic boundary conditions: {W}here
  do we stand?
\newblock {\em Milan J. Math.}, 83(2):237--278, 2015.

\bibitem{Gal-Grasselli}
C.G. Gal and M.~Grasselli.
\newblock The non-isothermal {A}llen-{C}ahn equation with dynamic boundary
  conditions.
\newblock {\em Discrete Contin. Dyn. Syst.}, 22(4):1009--1040, 2008.

\bibitem{Meyries}
C.G. Gal and M.~Meyries.
\newblock Nonlinear elliptic problems with dynamical boundary conditions of
  reactive and reactive-diffusive type.
\newblock {\em Proc. Lond. Math. Soc. (3)}, 108(6):1351--1380, 2014.

\bibitem{GK}
H.~Garcke and P.~Knopf.
\newblock {Weak solutions of the Cahn--Hilliard equation with dynamic boundary
  conditions: A gradient flow approach}.
\newblock {\em SIAM J. Math. Anal.}, 52(1):340--369, 2020.

\bibitem{Gazzola-Sweers}
F.~Gazzola and G.~Sweers.
\newblock On positivity for the biharmonic operator under {S}teklov boundary
  conditions.
\newblock {\em Arch. Ration. Mech. Anal.}, 188(3):399--427, 2008.

\bibitem{Girouard}
A.~Girouard and I.~Polterovich.
\newblock Spectral geometry of the {S}teklov problem (survey article).
\newblock {\em J. Spectr. Theory}, 7(2):321--359, 2017.

\bibitem{GMS}
G.R. Goldstein, A.~Miranville, and G.~Schimperna.
\newblock {A Cahn--Hilliard model in a domain with non permeable walls}.
\newblock {\em Physica D}, 240:754--766, 2011.

\bibitem{Gustafson}
K.~Gustafson and T.~Abe.
\newblock The third boundary condition---was it {R}obin's?
\newblock {\em Math. Intelligencer}, 20(1):63--71, 1998.

\bibitem{Gustafson2}
K.~Gustafson and T.~Abe.
\newblock ({V}ictor) {G}ustave {R}obin: 1855--1897.
\newblock {\em Math. Intelligencer}, 20(2):47--53, 1998.

\bibitem{Hsiao-Wendland}
G.~Hsiao and W.L. Wendland.
\newblock Boundary integral equations.
\newblock In {\em Applied Mathematical Sciences}. Springer-Verlag Berlin
  Heidelberg, 2008.

\bibitem{Kennedy}
J.B. Kennedy.
\newblock On the isoperimetric problem for the higher eigenvalues of the
  {R}obin and {W}entzell {L}aplacians.
\newblock {\em Z. Angew. Math. Phys.}, 61(5):781--792, 2010.

\bibitem{KL}
P.~Knopf and K.F. Lam.
\newblock {Convergence of a Robin boundary approximation for a Cahn--Hilliard
  system with dynamic boundary conditions}.
\newblock {\em Nonlinearity}, 33(8):4191--4235, 2020.

\bibitem{KLLM}
P.~Knopf, K.F. Lam, C.~Liu, and S.~Metzger.
\newblock {Phase-field dynamics with transfer of materials: The Cahn--Hillard
  equation with reaction rate dependent dynamic boundary conditions}.
\newblock \textit{Accepted in ESAIM Math. Model. Numer. Anal.}, Preprint:
  \href{https://arxiv.org/abs/2003.1298}{arXiv:2003.12983} [math.AP], 2020.

\bibitem{KS}
P.~Knopf and A.~Signori.
\newblock On the nonlocal {C}ahn--{H}illiard equation with nonlocal dynamic
  boundary condition and boundary penalization.
\newblock {\em J. Differential Equations}, 280:236--291, 2021.

\bibitem{Lam-Wu}
K.F. Lam and H.~Wu.
\newblock Convergence to equilibrium for a bulk-surface {A}llen-{C}ahn system
  coupled through a nonlinear {R}obin boundary condition.
\newblock {\em Discrete Contin. Dyn. Syst.}, 40(3):1847--1878, 2020.

\bibitem{LW}
C.~Liu and H.~Wu.
\newblock {An energetic variational approach for the {C}ahn-{H}illiard equation
  with dynamic boundary condition: model derivation and mathematical analysis}.
\newblock {\em Arch. Ration. Mech. Anal.}, 233(1):167--247, 2019.

\bibitem{Liu}
G.~Liu.
\newblock The {W}eyl-type asymptotic formula for biharmonic {S}teklov
  eigenvalues on {R}iemannian manifolds.
\newblock {\em Adv. Math.}, 228(4):2162--2217, 2011.

\bibitem{Matevossian}
H.~Matevossian.
\newblock On the mixed {D}irichlet-{S}teklov-type and {S}teklov-type biharmonic
  problems in weighted spaces.
\newblock {\em Math. Comput. Appl.}, 24(1):Paper No. 25, 9, 2019.

\bibitem{Mora}
D.~Mora, G.~Rivera, and R.~Rodr\'{\i}guez.
\newblock A virtual element method for the {S}teklov eigenvalue problem.
\newblock {\em Math. Models Methods Appl. Sci.}, 25(8):1421--1445, 2015.

\bibitem{Sprekels}
J.~Sprekels and H.~Wu.
\newblock A note on parabolic equation with nonlinear dynamical boundary
  condition.
\newblock {\em Nonlinear Anal.}, 72(6):3028--3048, 2010.

\bibitem{Steklov}
W.~Stekloff.
\newblock Sur les probl\`emes fondamentaux de la physique math\'{e}matique
  (suite et fin).
\newblock {\em Ann. Sci. \'{E}cole Norm. Sup. (3)}, 19:455--490, 1902.

\bibitem{Taylor}
M.E. Taylor.
\newblock {\em Partial differential equations {I}. {B}asic theory}, volume 115
  of {\em Applied Mathematical Sciences}.
\newblock Springer, New York, second edition, 2011.

\bibitem{Vazquez}
J.L. V\'{a}zquez and E.~Vitillaro.
\newblock On the {L}aplace equation with dynamical boundary conditions of
  reactive-diffusive type.
\newblock {\em J. Math. Anal. Appl.}, 354(2):674--688, 2009.

\bibitem{Xia}
C.~Xia and Q.~Wang.
\newblock Eigenvalues of the {W}entzell-{L}aplace operator and of the fourth
  order {S}teklov problems.
\newblock {\em J. Differential Equations}, 264(10):6486--6506, 2018.

\end{thebibliography}
%

\end{document}